\documentclass[12pt,high,twoside]{amsart}
\usepackage{amsmath,amsthm,amsfonts,amssymb,mathrsfs}
\usepackage{color}

\usepackage{amssymb,cite}
\usepackage[colorlinks,plainpages,citecolor=magenta, linkcolor=blue, backref]{hyperref}%,bookmarksnumbered,

\usepackage{eucal}

\usepackage{tikz}

\usepackage{hyperref}
\date{\today}

\usepackage{hyperref}
\input{xy}
 \xyoption{all}
 \xyoption{arc}

\newtheorem{theorem}{Theorem}

\newtheorem{proposition}{Proposition}%[section]
\newtheorem{corollary}{Corollary}%[section]

\newtheorem{lemma}{Lemma}%[section]
\theoremstyle{definition}

\begin{document}

\title[On the semigroup of injective endomorphisms of the semigroup $\boldsymbol{B}_{\omega}^{\mathscr{F}_n}$]{On the semigroup of injective endomorphisms of the semigroup $\boldsymbol{B}_{\omega}^{\mathscr{F}_n}$ which is generated by the family $\mathscr{F}_n$ of initial finite intervals of $\omega$}
\author{Oleg Gutik and Olha Popadiuk}
\address{Faculty of Mechanics and Mathematics,
Lviv University, Universytetska 1, Lviv, 79000, Ukraine}
\email{oleg.gutik@lnu.edu.ua, olha.popadiuk@lnu.edu.ua}



\begin{abstract}
In the paper we describe injective endomorphisms of the inverse semigroup $\boldsymbol{B}_{\omega}^{\mathscr{F}}$, which is introduced in the paper [O. Gutik and M. Mykhalenych, \emph{On some generalization of the bicyclic monoid}, Visnyk Lviv. Univ. Ser. Mech.-Mat.
\textbf{90} (2020), 5--19 (in Ukrainian)], in the case when the family $\mathscr{F}_n$ is generated by the set $\{0,1,\ldots,n\}$. In particular we show that the semigroup of injective endomorphisms of the semigroup $\boldsymbol{B}_{\omega}^{\mathscr{F}}$ is  isomorphic to $(\omega,+)$. Also we describe the structure of the semigroup  $\mathfrak{End}(\mathscr{B}_{\lambda})$ of all endomorphisms of the semigroup of $\lambda{\times}\lambda$-matrix units $\mathscr{B}_{\lambda}$.
\end{abstract}

\maketitle

%\tableofcontents

\section{Introduction, motivation and main definitions}

We shall follow the terminology of~\cite{Clifford-Preston-1961, Clifford-Preston-1967, Lawson-1998, Petrich-1984}. By $\omega$ we denote the set of all non-negative integers.

Let $\mathscr{P}(\omega)$ be  the family of all subsets of $\omega$. For any $F\in\mathscr{P}(\omega)$ and $n,m\in\omega$ we put $n-m+F=\{n-m+k\colon k\in F\}$ if $F\neq\varnothing$ and $n-m+\varnothing=\varnothing$. A subfamily $\mathscr{F}\subseteq\mathscr{P}(\omega)$ is called \emph{${\omega}$-closed} if $F_1\cap(-n+F_2)\in\mathscr{F}$ for all $n\in\omega$ and $F_1,F_2\in\mathscr{F}$.

We denote $[0;0]=\{0\}$ and $[0;k]=\{0,\ldots,k\}$ for any positive integer $k$. The set $[0;k]$, $k\in\omega$, is called an \emph{initial interval} of $\omega$.

A \emph{partially ordered set} (or shortly a \emph{poset}) $(X,\leqq)$ is the set $X$ with the reflexive, antisymmetric and transitive relation $\leqq$. In this case the relation $\leqq$ is called a partial order on $X$. A partially ordered set $(X,\leqq)$ is \emph{linearly ordered} or is a \emph{chain} if $x\leqq y$ or $y\leqq x$ for any $x,y\in X$. A map $f$ from a poset $(X,\leqq)$ onto a poset $(Y,\eqslantless)$ is said to be an order isomorphism if $f$ is bijective and $x\leqq y$ if and only if $f(x)\eqslantless f(y)$. A \emph{partial order isomorphism} $f$ from a poset $(X,\leqq)$ into a poset $(Y,\eqslantless)$ is an order isomorphism from a subset $A$ of a poset $(X,\leqq)$ onto a subset $B$ of a poset $(Y,\eqslantless)$. For any elements $x$ of a poset $(X,\leqq)$ we denote
\begin{equation*}
  {\uparrow_{\leqq}}x=\{y\in X\colon x\leqq y\}.
\end{equation*}

A nonempty set $S$ with a binary associative operation is called a \emph{semigroup}.
By $(\omega,+)$  we denote the set $\omega$ with the usual addition $(x,y)\mapsto x+y$.

A semigroup $S$ is called {\it inverse} if for any
element $x\in S$ there exists a unique $x^{-1}\in S$ such that
$xx^{-1}x=x$ and $x^{-1}xx^{-1}=x^{-1}$. The element $x^{-1}$ is
called the {\it inverse of} $x\in S$. If $S$ is an inverse
semigroup, then the mapping $\operatorname{inv}\colon S\to S$
which assigns to every element $x$ of $S$ its inverse element
$x^{-1}$ is called the {\it inversion}.

If $S$ is a semigroup, then we shall denote the subset of all
idempotents in $S$ by $E(S)$. If $S$ is an inverse semigroup, then
$E(S)$ is closed under multiplication and we shall refer to $E(S)$ as a
\emph{band} (or the \emph{band of} $S$). Then the semigroup
operation on $S$ determines the following partial order $\preccurlyeq$
on $E(S)$: $e\preccurlyeq f$ if and only if $ef=fe=e$. This order is
called the {\em natural partial order} on $E(S)$. A \emph{semilattice} is a commutative semigroup of idempotents. By $(\omega,\min)$  we denote the set $\omega$ with the semilattice operation $x\cdot y=\min\{x,y\}$.

%If $S$ is an inverse semigroup then the semigroup operation on $S$ determines the following partial order $\preccurlyeq$ on $S$: $s\preccurlyeq t$ if and only if there exists $e\in E(S)$ such that $s=te$. This order is called the {\em natural partial order} on $S$ \cite{Wagner-1952}.

For semigroups $S$ and $T$, a map $\mathfrak{h}\colon S\to T$ is called:
\begin{itemize}
  \item a \emph{homomorphism} if $\mathfrak{h}(s_1\cdot s_2)=\mathfrak{h}(s_1)\cdot \mathfrak{h}(s_2)$ for all $s_1,s_2\in S$;
  \item an \emph{annihilating homomorphism} if $\mathfrak{h}$ is a homomorphism and $\mathfrak{h}(s_1)=\mathfrak{h}(s_2)$ for all $s_1,s_2\in S$;
  \item an \emph{isomorphism} if $\mathfrak{h}\colon S\to T$ is a bijective homomorphism.
\end{itemize}
For a semigroup $S$ a homomorphism (an isomorphism) $\mathfrak{h}\colon S\to S$ is called an \emph{endomorphism} (\emph{automorphism}) of $S$. For simplicity of calculation the image of $s\in S$ under an endomorphism $\mathfrak{e}$ of a semigroup $S$ we shall denote  by $(s)\mathfrak{e}$.

A \emph{congruence} on a semigroup $S$ is an equivalence relation $\mathfrak{C}$ on $S$ such that $(s,t)\in\mathfrak{C}$ implies that $(as,at),(sb,tb)\in\mathfrak{C}$ for all $a,b\in S$. Every congruence $\mathfrak{C}$ on a semigroup $S$ generates the \emph{associated natural homomorphism} $\mathfrak{C}^\natural\colon S\to S/\mathfrak{C}$ which assigns to each element $s$ of $S$ its congruence class $[s]_\mathfrak{C}$ in the quotient semigroup $S/\mathfrak{C}$. Also every homomorphism $\mathfrak{h}\colon S\to T$ of semigroups $S$ and $T$ generates the congruence $\mathfrak{C}_\mathfrak{h}$ on $S$: $(s_1,s_2)\in\mathfrak{C}_\mathfrak{h}$ if and only if $(s_1)\mathfrak{h}=(s_2)\mathfrak{h}$.

A nonempty subset $I$ of a semigroup $S$ is called an \emph{ideal} of $S$ if $SIS=\{asb\colon s\in I, \; a,b\in S\}\subseteq I$. Every ideal $I$ of a semigroup $S$ generates the congruence $\mathfrak{C}_I=(I\times I)\cup\Delta_S$ on $S$, which is called the \emph{Rees congruence} on $S$.

Let $\mathscr{I}_\lambda$ denote the set of all partial one-to-one transformations of $\lambda$ together with the following semigroup operation:
\begin{equation*}
    x(\alpha\beta)=(x\alpha)\beta \quad \mbox{if} \quad
    x\in\operatorname{dom}(\alpha\beta)=\{
    y\in\operatorname{dom}\alpha\colon
    y\alpha\in\operatorname{dom}\beta\}, \qquad \mbox{for} \quad
    \alpha,\beta\in\mathscr{I}_\lambda.
\end{equation*}
The semigroup $\mathscr{I}_\lambda$ is called the \emph{symmetric
inverse semigroup} over the cardinal $\lambda$~(see \cite{Clifford-Preston-1961}). For any $\alpha\in\mathscr{I}_\lambda$ the cardinality of $\operatorname{dom}\alpha$ is called the \emph{rank} of $\alpha$ and it is denoted by $\operatorname{rank}\alpha$. The symmetric inverse semigroup was introduced by V.~V.~Wagner~\cite{Wagner-1952}
and it plays a major role in the theory of semigroups.

%For every $\alpha\in \mathscr{I}_\lambda$ we put $\operatorname{rank}\alpha=\left|\operatorname{dom}\alpha\right|$.

Put
$\mathscr{I}_\lambda^n=\{ \alpha\in\mathscr{I}_\lambda\colon
\operatorname{rank}\alpha\leqslant n\}$,
for $n=1,2,3,\ldots$. Obviously,
$\mathscr{I}_\lambda^n$ ($n=1,2,3,\ldots$) are inverse semigroups,
$\mathscr{I}_\lambda^n$ is an ideal of $\mathscr{I}_\lambda$, for each $n=1,2,3,\ldots$. The semigroup
$\mathscr{I}_\lambda^n$ is called the \emph{symmetric inverse semigroup of
finite transformations of the rank $\leqslant n$} \cite{Gutik-Reiter-2009}. By
\begin{equation*}
\left({%
\begin{smallmatrix}
  x_1 & x_2 & \cdots & x_n \\
  y_1 & y_2 & \cdots & y_n
\end{smallmatrix}
}\right)
\end{equation*}
we denote a partial one-to-one transformation which maps $x_1$ onto $y_1$, $x_2$ onto $y_2$, $\ldots$, and $x_n$ onto $y_n$. Obviously, in such case we have $x_i\neq x_j$ and $y_i\neq y_j$ for $i\neq j$ ($i,j=1,2,3,\ldots,n$). The empty partial map $\varnothing\colon \lambda\rightharpoonup\lambda$ is denoted by $\boldsymbol{0}$. It is obvious that $\boldsymbol{0}$ is zero of the semigroup $\mathscr{I}_\lambda^n$.

For a partially ordered set $(P, \leqq)$, a subset $X$ of $P$ is called \emph{order-convex}, if $x\leqq z\leqq y$ and ${x, y}\subset X$ implies that $z\in X$, for all $x, y, z\in P$ \cite{Harzheim=2005}. It is obvious that the set of all partial order isomorphisms between convex subsets of $(\omega,\leqslant)$ under the composition of partial self-maps forms an inverse subsemigroup of the symmetric inverse semigroup $\mathscr{I}_\omega$ over the set $\omega$. We denote this semigroup by $\mathscr{I}_\omega(\overrightarrow{\mathrm{conv}})$. We put $\mathscr{I}_\omega^n(\overrightarrow{\mathrm{conv}})=\mathscr{I}_\omega(\overrightarrow{\mathrm{conv}})\cap\mathscr{I}_\omega^n$ and it is obvious that $\mathscr{I}_\omega^n(\overrightarrow{\mathrm{conv}})$ is closed under the semigroup operation of $\mathscr{I}_\omega^n$ and the semigroup $\mathscr{I}_\omega^n(\overrightarrow{\mathrm{conv}})$ is called the \emph{inverse semigroup of
convex order isomorphisms of $(\omega,\leqslant)$ of the rank $\leqslant n$}. Obviously that every non-zero element of the semigroup $\mathscr{I}_\omega^n(\overrightarrow{\mathrm{conv}})$ of the rank  $k\leqslant n$ has a form
\begin{equation*}
\left({%
\begin{smallmatrix}
  i & i+1 & \cdots & i+k-1 \\
  j & j+1 & \cdots & j+k-1
\end{smallmatrix}
}\right)
\end{equation*}
for some $i,j\in\omega$.

The bicyclic monoid ${\mathscr{C}}(p,q)$ is the semigroup with the identity $1$ generated by two elements $p$ and $q$ subjected only to the condition $pq=1$. The semigroup operation on ${\mathscr{C}}(p,q)$ is determined as follows:
\begin{equation*}
    q^kp^l\cdot q^mp^n=q^{k+m-\min\{l,m\}}p^{l+n-\min\{l,m\}}.
\end{equation*}
It is well known that the bicyclic monoid ${\mathscr{C}}(p,q)$ is a bisimple (and hence simple) combinatorial $E$-unitary inverse semigroup and every non-trivial congruence on ${\mathscr{C}}(p,q)$ is a group congruence \cite{Clifford-Preston-1961}.

On the set $\boldsymbol{B}_{\omega}=\omega\times\omega$ we define the semigroup operation ``$\cdot$'' in the following way
\begin{equation*}
  (i_1,j_1)\cdot(i_2,j_2)=
  \left\{
    \begin{array}{ll}
      (i_1-j_1+i_2,j_2), & \hbox{if~} j_1\leqslant i_2;\\
      (i_1,j_1-i_2+j_2), & \hbox{if~} j_1\geqslant i_2.
    \end{array}
  \right.
\end{equation*}
It is well known that the semigroup $\boldsymbol{B}_{\omega}$ is isomorphic to the bicyclic monoid by the mapping $\mathfrak{h}\colon \mathscr{C}(p,q)\to \boldsymbol{B}_{\omega}$, $q^kp^l\mapsto (k,l)$ (see: \cite[Section~1.12]{Clifford-Preston-1961} or \cite[Exercise IV.1.11$(ii)$]{Petrich-1984}).

%A {\it topological} ({\it semitopological}) {\it semigroup} is a topological space together with a continuous (separately continuous) semigroup operation. If $S$ is a~semigroup and $\tau$ is a topology on $S$ such that $(S,\tau)$ is a topological semigroup, then we shall call $\tau$ a \emph{semigroup} \emph{topology} on $S$, and if $\tau$ is a topology on $S$ such that $(S,\tau)$ is a semitopological semigroup, then we shall call $\tau$ a \emph{shift-continuous} \emph{topology} on~$S$. %An inverse topological semigroup with the continuous inversion is called a \emph{topological inverse semigroup}. If $S$ is an~inverse semigroup and $\tau$ is a topology on $S$ such that $(S,\tau)$ is a topological inverse semigroup, then we shall call $\tau$ a \emph{semigroup inverse topology} on $S$.

Next we shall describe the construction which is introduced in \cite{Gutik-Mykhalenych=2020}.

Let $\boldsymbol{B}_{\omega}$ be the bicyclic monoid and $\mathscr{F}$ be an ${\omega}$-closed subfamily of $\mathscr{P}(\omega)$. On the set $\boldsymbol{B}_{\omega}\times\mathscr{F}$ we define the semigroup operation ``$\cdot$'' in the following way
\begin{equation*}
  (i_1,j_1,F_1)\cdot(i_2,j_2,F_2)=
  \left\{
    \begin{array}{ll}
      (i_1-j_1+i_2,j_2,(j_1-i_2+F_1)\cap F_2), & \hbox{if~} j_1\leqslant i_2;\\
      (i_1,j_1-i_2+j_2,F_1\cap (i_2-j_1+F_2)), & \hbox{if~} j_1\geqslant i_2.
    \end{array}
  \right.
\end{equation*}
In \cite{Gutik-Mykhalenych=2020} is proved that if the family $\mathscr{F}\subseteq\mathscr{P}(\omega)$ is ${\omega}$-closed then $(\boldsymbol{B}_{\omega}\times\mathscr{F},\cdot)$ is a semigroup. Moreover, if an ${\omega}$-closed family  $\mathscr{F}\subseteq\mathscr{P}(\omega)$ contains the empty set $\varnothing$ then the set
$ %\begin{equation*}
  \boldsymbol{I}=\{(i,j,\varnothing)\colon i,j\in\omega\}
$ %\end{equation*}
is an ideal of the semigroup $(\boldsymbol{B}_{\omega}\times\mathscr{F},\cdot)$. For any ${\omega}$-closed family $\mathscr{F}\subseteq\mathscr{P}(\omega)$ the following semigroup
\begin{equation*}
  \boldsymbol{B}_{\omega}^{\mathscr{F}}=
\left\{
  \begin{array}{ll}
    (\boldsymbol{B}_{\omega}\times\mathscr{F},\cdot)/\boldsymbol{I}, & \hbox{if~} \varnothing\in\mathscr{F};\\
    (\boldsymbol{B}_{\omega}\times\mathscr{F},\cdot), & \hbox{if~} \varnothing\notin\mathscr{F}
  \end{array}
\right.
\end{equation*}
is defined in \cite{Gutik-Mykhalenych=2020}. The semigroup $\boldsymbol{B}_{\omega}^{\mathscr{F}}$ generalizes the bicyclic monoid and the countable semigroup of matrix units. It is proved in \cite{Gutik-Mykhalenych=2020} that $\boldsymbol{B}_{\omega}^{\mathscr{F}}$ is a combinatorial inverse semigroup and Green's relations, the natural partial order on $\boldsymbol{B}_{\omega}^{\mathscr{F}}$ and its set of idempotents are described. The criteria of simplicity, $0$-simplicity, bisimplicity, $0$-bisimplicity of the semigroup $\boldsymbol{B}_{\omega}^{\mathscr{F}}$ and when $\boldsymbol{B}_{\omega}^{\mathscr{F}}$ has the identity, is isomorphic to the bicyclic semigroup or the countable semigroup of matrix units are given. In particularly in \cite{Gutik-Mykhalenych=2020} is proved that the semigroup $\boldsymbol{B}_{\omega}^{\mathscr{F}}$ is isomorphic to the semigrpoup of ${\omega}{\times}{\omega}$-matrix units if and only if $\mathscr{F}$ consists of a singleton set and the empty set.

The semigroup $\boldsymbol{B}_{\omega}^{\mathscr{F}}$ in the case when the family $\mathscr{F}$ consists of the empty set and some singleton subsets of $\omega$ is studied in \cite{Gutik-Lysetska=2021}. It is proved that the semigroup $\boldsymbol{B}_{\omega}^{\mathscr{F}}$ is isomorphic to the  subsemigroup $\mathscr{B}_{\omega}^{\Rsh}(\boldsymbol{F}_{\min})$ of the Brandt $\omega$-extension of the subsemilattice $(\boldsymbol{F},\min)$ of $(\omega,\min)$, where $\boldsymbol{F}=\bigcup\mathscr{F}$. Also topologizations of the semigroup  $\boldsymbol{B}_{\omega}^{\mathscr{F}}$ and its closure in semitopological semigroups are studied.

For any $n\in\omega$ we put $\mathscr{F}_n=\left\{[0;k]\colon k=0,\ldots,n\right\}$. It is obvious that $\mathscr{F}_n$ is an $\omega$-closed family of $\omega$.

In the paper \cite{Gutik-Popadiuk=2022} we study the semigroup $\boldsymbol{B}_{\omega}^{\mathscr{F}_n}$. It is shown that the Green relations $\mathscr{D}$ and $\mathscr{J}$ coincide in $\boldsymbol{B}_{\omega}^{\mathscr{F}_n}$, the semigroup $\boldsymbol{B}_{\omega}^{\mathscr{F}_n}$ is isomorphic to the semigroup $\mathscr{I}_\omega^{n+1}(\overrightarrow{\mathrm{conv}})$, and  $\boldsymbol{B}_{\omega}^{\mathscr{F}_n}$ admits only Rees congruences. Also in \cite{Gutik-Popadiuk=2022}, we study shift-continuous topologies of the semigroup $\boldsymbol{B}_{\omega}^{\mathscr{F}_n}$. In particular we prove that for any shift-continuous $T_1$-topology $\tau$ on the semigroup $\boldsymbol{B}_{\omega}^{\mathscr{F}_n}$ every non-zero element of $\boldsymbol{B}_{\omega}^{\mathscr{F}_n}$ is an isolated point of $(\boldsymbol{B}_{\omega}^{\mathscr{F}_n},\tau)$, $\boldsymbol{B}_{\omega}^{\mathscr{F}_n}$ admits the unique compact shift-continuous $T_1$-topology, and every $\omega_{\mathfrak{d}}$-compact shift-continuous $T_1$-topology is compact. We describe the closure of the semigroup $\boldsymbol{B}_{\omega}^{\mathscr{F}_n}$ in a Hausdorff semitopological semigroup and prove the criterium when a topological inverse semigroup $\boldsymbol{B}_{\omega}^{\mathscr{F}_n}$ is $H$-closed in the class of Hausdorff topological semigroups.

Surprisingly, not so many articles are devoted to endomorphisms and automorphisms of semigroups. In particular, in \cite{Araujo-Fernandes-Jesus-Maltcev-Mitchell=2011} the authors propose a general recipe for calculating the automorphism groups of semigroups consisting of partial endomorphisms of relational structures over a finite set with a single $m$-ary relation for any positive integer $m$, which determine the automorphism groups of the following semigroups: the full transformation semigroup, the partial transformation semigroup, and the symmetric inverse semigroup, the wreath product of two full transformation semigroups, the partial endomorphisms of any partially ordered set, the full spectrum of semigroups of partial mappings preserving or reversing a linear or circular order. 
In the paper \cite{Fernandes-Jesus-Maltcev-Mitchell=2010} the authors characterize the endomorphisms of the semigroup of all order-
preserving mappings on a finite chain. 
In  \cite{Fernandes-Santos=2019} Fernandes and Santos characterize the monoids of endomorphisms of the semigroup of all order-preserving partial transformations and of the semigroup of all order-preserving
partial permutations of a finite chain. Also the semigroups of a finite chain are described in \cite{Lavers-Solomon=1999, Popova=1962}. 
Endomorphisms and automorphisms of other types of semigroups are studied in \cite{Aizenshtat=1962, Fitzpatrick-Symons=1974, Gutik-Pozdniakova=2022, Gutik-Prokhorenkova-Sekh=2021, Levi-O'Meara-Wood=1986, Magill=1966, Mazorchuk=2002, Schein-Teclezghi=1997, Schein-Teclezghi=1998, Sullivan=1975, Zhuchok-2010} and other papers.

This paper is a continuations of the investigation which are presented in \cite{Gutik-Popadiuk=2022}.
Here we describe injective endomorphisms of the semigroup $\mathscr{I}_{\omega}^{n}(\overrightarrow{\operatorname{conv}})$ for a positive integer $n\geqslant 2$. In particular we show that for $n\geqslant 2$ the semigroup of injective endomorphisms of the semigroup $\boldsymbol{B}_{\omega}^{\mathscr{F}_n}$ is  isomorphic to $(\omega,+)$. Also we describes the structure of the semigroup  $\mathfrak{End}(\mathscr{B}_{\lambda})$ of all endomorphisms of the semigroup of $\lambda{\times}\lambda$-matrix units $\mathscr{B}_{\lambda}$.

\section{On injective endomorphisms of the semigroup $\boldsymbol{B}_{\omega}^{\mathscr{F}_n}$}

\begin{proposition}\label{proposition-2.1}
For any non-negative integer $n$ and arbitrary $p\in\omega$ the map $\mathfrak{e}_p\colon \boldsymbol{B}_{\omega}^{\mathscr{F}_n}\to \boldsymbol{B}_{\omega}^{\mathscr{F}_n}$ defined by the formulae $(\boldsymbol{0})\mathfrak{e}_p=\boldsymbol{0}$ and
\begin{equation*}
(i,j,[0;k])\mathfrak{e}_p=
    (p+i,p+j,[0;k]),
\end{equation*}
is an endomorphism of the semigroup $\boldsymbol{B}_{\omega}^{\mathscr{F}_n}$.
\end{proposition}

\begin{proof}
It is obvious that $(\boldsymbol{0})\mathfrak{e}_p\cdot(\boldsymbol{0})\mathfrak{e}_p=\boldsymbol{0}\cdot\boldsymbol{0}= \boldsymbol{0}=(\boldsymbol{0})\mathfrak{e}_p=(\boldsymbol{0}\cdot\boldsymbol{0})\mathfrak{e}_p$ and
\begin{align*}
  (\boldsymbol{0})\mathfrak{e}_p\cdot(i,j,[0;k])\mathfrak{e}_p&=\boldsymbol{0}\cdot(p+i,p+j,[0;k])= \boldsymbol{0}=(\boldsymbol{0})\mathfrak{e}_p=(\boldsymbol{0}\cdot(i,j,[0;k]))\mathfrak{e}_p;\\
  (i,j,[0;k])\mathfrak{e}_p\cdot(\boldsymbol{0})\mathfrak{e}_p&=(p+i,p+j,[0;k])\cdot\boldsymbol{0}=  \boldsymbol{0}=(\boldsymbol{0})\mathfrak{e}_p=((i,j,[0;k])\cdot\boldsymbol{0})\mathfrak{e}_p,
\end{align*}
for any non-zero element $(i,j,[0;k])$ of the semigroup $\boldsymbol{B}_{\omega}^{\mathscr{F}_n}$. Also, for any non-zero elements $(i_1,j_1,[0;k_1])$ and $(i_2,j_2,[0;k_2])$ of the semigroup $\boldsymbol{B}_{\omega}^{\mathscr{F}_n}$ we have that
\begin{align*}
  (i_1,j_1,[0;k_1])\mathfrak{e}_p&\cdot(i_2,j_2,[0;k_2])\mathfrak{e}_p=  (p+i_1,p+j_1,[0;k_1])\cdot(p+i_2,p+j_2,[0;k_2])=\\
  =&
  \left\{
    \begin{array}{ll}
      (p{+}i_1-(p{+}j_1){+}p{+}i_2,p{+}j_2,(p{+}j_1-(p{+}i_2){+}[0;k_1])\cap[0;k_2]), & \hbox{if~} p{+}j_1<p{+}i_2;\\
      (p+i_1,p+j_2,[0;k_1]\cap[0;k_2]),                                               & \hbox{if~} p{+}j_1=p{+}i_2;\\
      (p{+}i_1,p{+}j_1-(p{+}i_2){+}p{+}j_2,[0;k_1]\cap(p{+}i_2-(p{+}j_1){+}[0;k_2])), & \hbox{if~} p{+}j_1>p{+}i_2
    \end{array}
  \right.{=}\\
  =&
  \left\{
    \begin{array}{ll}
      (p+i_1-j_1+i_2,p+j_2,(j_1-i_2+[0;k_1])\cap[0;k_2]), & \hbox{if~} j_1<i_2;\\
      (p+i_1,p+j_2,[0;k_1]\cap[0;k_2]),                   & \hbox{if~} j_1=i_2;\\
      (p+i_1,p+j_1-i_2+j_2,[0;k_1]\cap(i_2-j_1+[0;k_2])), & \hbox{if~} j_1>i_2
    \end{array}
  \right.
\end{align*}
and
\begin{align*}
  ((i_1,j_1,[0;k_1])\cdot(i_2,j_2,[0;k_2]))\mathfrak{e}_p&=
  \left\{
    \begin{array}{ll}
      (i_1-j_1+i_2,j_2,(j_1-i_2+[0;k_1])\cap[0;k_2])\mathfrak{e}_p, & \hbox{if~} j_1<i_2;\\
      (i_1,j_2,[0;k_1]\cap[0;k_2])\mathfrak{e}_p,                   & \hbox{if~} j_1=i_2;\\
      (i_1,j_1-i_2+j_2,[0;k_2]\cap(i_2-j_1+[0;k_2]))\mathfrak{e}_p, & \hbox{if~} j_1>i_2
    \end{array}
  \right.= \\
  &=
  \left\{
    \begin{array}{ll}
      (p+i_1-j_1+i_2,p+j_2,(j_1-i_2+[0;k_1])\cap[0;k_2]), & \hbox{if~} j_1<i_2;\\
      (p+i_1,p+j_2,[0;k_1]\cap[0;k_2]),                   & \hbox{if~} j_1=i_2;\\
      (p+i_1,p+j_1-i_2+j_2,[0;k_2]\cap(i_2-j_1+[0;k_2])), & \hbox{if~} j_1>i_2,
    \end{array}
  \right.
\end{align*}
and hence the  map $\mathfrak{e}_p$ is an endomorphism of the semigroup $\boldsymbol{B}_{\omega}^{\mathscr{F}_n}$.
\end{proof}

By Theorem~1 of \cite{Gutik-Popadiuk=2022} for any $n\in\omega$ the semigroup $\boldsymbol{B}_{\omega}^{\mathscr{F}_n}$ is isomorphic to the semigroup $\mathscr{I}_\omega^{n+1}(\overrightarrow{\mathrm{conv}})$ by the mapping $\mathfrak{I}\colon \boldsymbol{B}_{\omega}^{\mathscr{F}_n}\to \mathscr{I}_\omega^{n+1}(\overrightarrow{\operatorname{conv}})$, defined by the formulae $(\boldsymbol{0})\mathfrak{I}=\boldsymbol{0}$ and
\begin{equation*}
(i,j,[0;k])\mathfrak{I}=
  \left(%
\begin{smallmatrix}
  i & i+1 & \cdots & i+k \\
  j & j+1 & \cdots & j+k \\
\end{smallmatrix}%
\right).
\end{equation*}
This and Proposition~\ref{proposition-2.1} imply the following corollary.

\begin{corollary}\label{corollary-2.2}
For any positive integer $n$ and arbitrary $p\in\omega$ the map $\mathfrak{e}_p\colon \mathscr{I}_\omega^{n}\to\mathscr{I}_\omega^{n}$ defined by the formulae $(\boldsymbol{0})\mathfrak{e}_p=\boldsymbol{0}$ and
\begin{equation*}
\left(%
\begin{smallmatrix}
  i & i+1 & \cdots & i+k \\
  j & j+1 & \cdots & j+k \\
\end{smallmatrix}%
\right)\mathfrak{e}_p=
\left(%
\begin{smallmatrix}
  p+i & p+i+1 & \cdots & p+i+k \\
  p+j & p+j+1 & \cdots & p+j+k \\
\end{smallmatrix}%
\right), \qquad k=0,\ldots,n-1,
\end{equation*}
is an endomorphism of the semigroup $\mathscr{I}_\omega^{n}$.
\end{corollary}

Later we shall study endomorphisms of the semigroup $\mathscr{I}_\omega^{n}(\overrightarrow{\mathrm{conv}})$.

\begin{lemma}\label{lemma-2.3}
Let $n$ be any positive integer and $\frak{a}$ be an arbitrary non-annihilating endomorphism  of the semigroup $\mathscr{I}_\omega^{n}(\overrightarrow{\mathrm{conv}})$. Then $(\boldsymbol{0})\frak{a}=\boldsymbol{0}$.
\end{lemma}

\begin{proof}
Since $\boldsymbol{0}$ is an idempotent of $\mathscr{I}_\omega^{n}(\overrightarrow{\mathrm{conv}})$, so is the image $(\boldsymbol{0})\frak{a}$. Suppose to the contrary that $(\boldsymbol{0})\frak{a}=\boldsymbol{e}\neq\boldsymbol{0}$. By Theorem~3 from \cite{Gutik-Popadiuk=2022} the image of $\mathscr{I}_\omega^{n}(\overrightarrow{\mathrm{conv}})$ under the endomorphism $\frak{a}$ is isomorphic to the semigroup $\mathscr{I}_\omega^{m}(\overrightarrow{\mathrm{conv}})$ for some positive integer $m\leqslant n$. Hence the subsemigroup $(\mathscr{I}_\omega^{n}(\overrightarrow{\mathrm{conv}}))\frak{a}$ of $\mathscr{I}_\omega^{n}(\overrightarrow{\mathrm{conv}})$ has infinitely many  idempotents. But by Theorem~1 and Lemma~1 from \cite{Gutik-Popadiuk=2022} the set ${\uparrow_{\preccurlyeq}}\boldsymbol{e}$ is finite, a contradiction. The obtained contradiction implies the equality $(\boldsymbol{0})\frak{a}=\boldsymbol{0}$.
\end{proof}

Lemma~\ref{lemma-2.3} implies the following corollary.

\begin{corollary}\label{corollary-2.4}
Let $n$ be any positive integer and $\frak{a}$ be an endomorphism  of the semigroup $\mathscr{I}_\omega^{n}(\overrightarrow{\mathrm{conv}})$. If $(\boldsymbol{0})\frak{a}\neq\boldsymbol{0}$ then  $\frak{a}$ is annihilating.
\end{corollary}

\begin{lemma}\label{lemma-2.5}
Let $n$ be any positive integer $\geqslant 2$ and $\frak{a}$ be an arbitrary non-annihilating endomorphism  of the semigroup $\mathscr{I}_\omega^{n}(\overrightarrow{\mathrm{conv}})$. If
$\left(
\begin{smallmatrix}
  0\\
  0
\end{smallmatrix}
\right)\frak{a}=
\left(
\begin{smallmatrix}
  0\\
  0
\end{smallmatrix}
\right)$
then $\frak{a}$ is the identity automorphism of $\mathscr{I}_\omega^{n}(\overrightarrow{\mathrm{conv}})$.
\end{lemma}

\begin{proof}
First we shall show that the restriction of the endomorphism $\frak{a}$ onto the band $E(\mathscr{I}_\omega^{n}(\overrightarrow{\mathrm{conv}}))$   is the identity map of $E(\mathscr{I}_\omega^{n}(\overrightarrow{\mathrm{conv}}))$.

The definition of the natural partial order $\preccurlyeq$ on $E(\mathscr{I}_\omega^{n}(\overrightarrow{\mathrm{conv}}))$ implies that
\begin{equation*}
  {\uparrow_{\preccurlyeq}}
  \left(
\begin{smallmatrix}
  0\\
  0
\end{smallmatrix}
\right)=
\left\{
\left(
\begin{smallmatrix}
  0\\
  0
\end{smallmatrix}
\right),
\left(
\begin{smallmatrix}
  0 & 1\\
  0 & 1
\end{smallmatrix}
\right),
\ldots,
\left(
\begin{smallmatrix}
  0 & 1 & \cdots & n-1\\
  0 & 1 & \cdots & n-1
\end{smallmatrix}
\right)
\right\}.
\end{equation*}
By Proposition~1.14.21(6) of \cite{Lawson-1998} every homomorphism of inverse semigroups preserves the natural partial order, and hence
$
\left(
\begin{smallmatrix}
  0\\
  0
\end{smallmatrix}
\right)\frak{a}
\preccurlyeq
\left(
\begin{smallmatrix}
  0 & 1\\
  0 & 1
\end{smallmatrix}
\right)\frak{a}
$, because
$
\left(
\begin{smallmatrix}
  0\\
  0
\end{smallmatrix}
\right)
\preccurlyeq
\left(
\begin{smallmatrix}
  0 & 1\\
  0 & 1
\end{smallmatrix}
\right)
$.
Also, by Proposition~4 of \cite{Gutik-Popadiuk=2022} every congruence on the semigroup $\mathscr{I}_\omega^{n}(\overrightarrow{\mathrm{conv}})$ is Rees, which implies that
$
\left(
\begin{smallmatrix}
  0\\
  0
\end{smallmatrix}
\right)\frak{a}
\neq
\left(
\begin{smallmatrix}
  0 & 1\\
  0 & 1
\end{smallmatrix}
\right)\frak{a}
$.
Hence we obtain that
$
\left(
\begin{smallmatrix}
  0 & 1\\
  0 & 1
\end{smallmatrix}
\right)\frak{a}=
\left(
\begin{smallmatrix}
  0 & 1\\
  0 & 1
\end{smallmatrix}
\right)
$. Similarly by induction we get that
$
\left(
\begin{smallmatrix}
  0 & 1 & \cdots & k\\
  0 & 1 & \cdots & k
\end{smallmatrix}
\right)\frak{a}=
\left(
\begin{smallmatrix}
  0 & 1 & \cdots & k\\
  0 & 1 & \cdots & k
\end{smallmatrix}
\right)
$
for any $k=2,\ldots, n-1$.

The definition of the natural partial order $\preccurlyeq$ on $E(\mathscr{I}_\omega^{n}(\overrightarrow{\mathrm{conv}}))$ implies that
$
\boldsymbol{0}\preccurlyeq
\left(
\begin{smallmatrix}
  1\\
  1
\end{smallmatrix}
\right)
\preccurlyeq
\left(
\begin{smallmatrix}
  0 & 1\\
  0 & 1
\end{smallmatrix}
\right)
$.
The above part of the proof, Lemma~\ref{lemma-2.3} and Proposition~4 of \cite{Gutik-Popadiuk=2022} imply that
\begin{equation*}
\boldsymbol{0}=(\boldsymbol{0})\frak{a}\preccurlyeq
\left(
\begin{smallmatrix}
  1\\
  1
\end{smallmatrix}
\right)\frak{a}
\preccurlyeq
\left(
\begin{smallmatrix}
  0 & 1\\
  0 & 1
\end{smallmatrix}
\right)\frak{a}=
\left(
\begin{smallmatrix}
  0 & 1\\
  0 & 1
\end{smallmatrix}
\right)
.
\end{equation*}
Again, by the definition of the natural partial order $\preccurlyeq$ on $E(\mathscr{I}_\omega^{n}(\overrightarrow{\mathrm{conv}}))$ we have that the inequalities $
\boldsymbol{0}\preccurlyeq
\boldsymbol{x}
\preccurlyeq
\left(
\begin{smallmatrix}
  0 & 1\\
  0 & 1
\end{smallmatrix}
\right)
$
have two solutions either $\boldsymbol{x}=
\left(
\begin{smallmatrix}
  0\\
  0
\end{smallmatrix}
\right)$
or
$\boldsymbol{x}=
\left(
\begin{smallmatrix}
  1\\
  1
\end{smallmatrix}
\right)$.
Then Proposition~4 of \cite{Gutik-Popadiuk=2022} implies that
$
\left(
\begin{smallmatrix}
  1\\
  1
\end{smallmatrix}
\right)\frak{a}=
\left(
\begin{smallmatrix}
  1\\
  1
\end{smallmatrix}
\right)
$.
Similar arguments and the following conditions
\begin{equation*}
  \left(
\begin{smallmatrix}
  1\\
  1
\end{smallmatrix}
\right)=
\left(
\begin{smallmatrix}
  1\\
  1
\end{smallmatrix}
\right)\frak{a}\preccurlyeq
\left(
\begin{smallmatrix}
  1 & 2\\
  1 & 2
\end{smallmatrix}
\right)\frak{a}\preccurlyeq
\left(
\begin{smallmatrix}
  0 & 1 & 2\\
  0 & 1 & 2
\end{smallmatrix}
\right)\frak{a}=
\left(
\begin{smallmatrix}
  0 & 1 & 2\\
  0 & 1 & 2
\end{smallmatrix}
\right)
\end{equation*}
imply that
$
\left(
\begin{smallmatrix}
  1 & 2\\
  1 & 2
\end{smallmatrix}
\right)\frak{a}=
\left(
\begin{smallmatrix}
  1 & 2\\
  1 & 2
\end{smallmatrix}
\right)
$. Next by induction we get that
$
\left(
\begin{smallmatrix}
  1 & 2 & \cdots & k+1\\
  1 & 2 & \cdots & k+1
\end{smallmatrix}
\right)\frak{a}=
\left(
\begin{smallmatrix}
  1 & 2 & \cdots & k+1\\
  1 & 2 & \cdots & k+1
\end{smallmatrix}
\right)
$
for any $k=2,\ldots, n-1$.

We observe that the proof of the step of induction: \emph{the equalities }
\begin{equation*}
\left(
\begin{smallmatrix}
  p\\
  p
\end{smallmatrix}
\right)=
\left(
\begin{smallmatrix}
  p\\
  p
\end{smallmatrix}
\right)\frak{a}, \quad
\left(
\begin{smallmatrix}
  p & p+1\\
  p & p+1
\end{smallmatrix}
\right)=
\left(
\begin{smallmatrix}
  p & p+1\\
  p & p+1
\end{smallmatrix}
\right)\frak{a}, \quad \ldots \quad, \quad
\left(
\begin{smallmatrix}
  p & p+1 & \cdots & p+n-1\\
  p & p+1 & \cdots & p+n-1
\end{smallmatrix}
\right)=
\left(
\begin{smallmatrix}
  p & p+1 & \cdots & p+n-1\\
  p & p+1 & \cdots & p+n-1
\end{smallmatrix}
\right)\frak{a}
\end{equation*}
\emph{hold for $p\leqslant m$, imply that these equalities hold for $p= m+1$,} is similar to the above part of the proof.

Fix an arbitrary $\boldsymbol{x}\in \mathscr{I}_\omega^{n}(\overrightarrow{\mathrm{conv}})\setminus E(\mathscr{I}_\omega^{n}(\overrightarrow{\mathrm{conv}}))$ with $\operatorname{rank}\boldsymbol{x}=k$, $k=1,\ldots,n$. Since $\boldsymbol{x}$ is a partial convex order isomorphism of $(\omega,\leqslant)$, there exist $s,t\in\omega$ such that $\boldsymbol{x}=
\left(
\begin{smallmatrix}
  s & s+1 & \cdots & s+k-1\\
  t & t+1 & \cdots & t+k-1
\end{smallmatrix}
\right)
$. Since $\boldsymbol{x}\boldsymbol{x}^{-1}, \boldsymbol{x}^{-1}\boldsymbol{x}\in E(\mathscr{I}_\omega^{n}(\overrightarrow{\mathrm{conv}}))$, by Proposition~1.14.21(1) of \cite{Lawson-1998} we have that
\begin{align*}
  (\boldsymbol{x})\frak{a} \cdot ((\boldsymbol{x})\frak{a})^{-1}&= (\boldsymbol{x})\frak{a} \cdot (\boldsymbol{x}^{-1})\frak{a}=\\
   &=(\boldsymbol{x}\boldsymbol{x}^{-1})\frak{a}=\\
   &=\boldsymbol{x}\boldsymbol{x}^{-1}=\\
   &=
\left(
\begin{smallmatrix}
  s & s+1 & \cdots & s+k-1\\
  t & t+1 & \cdots & t+k-1
\end{smallmatrix}
\right)
\left(
\begin{smallmatrix}
  s & s+1 & \cdots & s+k-1\\
  t & t+1 & \cdots & t+k-1
\end{smallmatrix}
\right)^{-1}=   \\
   &=
\left(
\begin{smallmatrix}
  s & s+1 & \cdots & s+k-1\\
  t & t+1 & \cdots & t+k-1
\end{smallmatrix}
\right)
\left(
\begin{smallmatrix}
  t & t+1 & \cdots & t+k-1\\
  s & s+1 & \cdots & s+k-1
\end{smallmatrix}
\right)=   \\
 &=
\left(
\begin{smallmatrix}
  s & s+1 & \cdots & s+k-1\\
  s & s+1 & \cdots & s+k-1
\end{smallmatrix}
\right)
\end{align*}
and
\begin{align*}
  ((\boldsymbol{x})\frak{a})^{-1}\cdot(\boldsymbol{x})\frak{a}&= (\boldsymbol{x}^{-1})\frak{a} \cdot (\boldsymbol{x})\frak{a}=\\
   &=(\boldsymbol{x}^{-1}\boldsymbol{x})\frak{a}=\\
   &=\boldsymbol{x}^{-1}\boldsymbol{x}=\\
   &=
\left(
\begin{smallmatrix}
  s & s+1 & \cdots & s+k-1\\
  t & t+1 & \cdots & t+k-1
\end{smallmatrix}
\right)^{-1}
\left(
\begin{smallmatrix}
  s & s+1 & \cdots & s+k-1\\
  t & t+1 & \cdots & t+k-1
\end{smallmatrix}
\right)=   \\
   &=
\left(
\begin{smallmatrix}
  t & t+1 & \cdots & t+k-1\\
  s & s+1 & \cdots & s+k-1
\end{smallmatrix}
\right)
\left(
\begin{smallmatrix}
  s & s+1 & \cdots & s+k-1\\
  t & t+1 & \cdots & t+k-1
\end{smallmatrix}
\right)=   \\
 &=
\left(
\begin{smallmatrix}
  t & t+1 & \cdots & t+k-1\\
  t & t+1 & \cdots & t+k-1
\end{smallmatrix}
\right).
\end{align*}
The above equalities imply that
\begin{equation*}
\operatorname{dom}((\boldsymbol{x})\frak{a})=\operatorname{dom}((\boldsymbol{x})\frak{a}\cdot ((\boldsymbol{x})\frak{a})^{-1})=\{s, \ldots, s+k-1\}
\end{equation*}
and
\begin{equation*}
\operatorname{ran}((\boldsymbol{x})\frak{a})=\operatorname{dom}(((\boldsymbol{x})\frak{a})^{-1}\cdot(\boldsymbol{x})\frak{a})=\{t, \ldots, t+k-1\}.
\end{equation*}
Since $(\boldsymbol{x})\frak{a}$ is a partial convex order isomorphism of $(\omega,\leqslant)$, we get that $(\boldsymbol{x})\frak{a}=
\left(
\begin{smallmatrix}
  s & s+1 & \cdots & s+k-1\\
  t & t+1 & \cdots & t+k-1
\end{smallmatrix}
\right)
$, which completes the proof of the lemma.
\end{proof}

For visually simplify of the proof of Theorem~\ref{theorem-2.6}, we schematically present the natural partial order on the semilattice $E(\mathscr{I}_\omega^{n}(\overrightarrow{\mathrm{conv}}))$ on Figure~\ref{fig-2.1}.

\begin{figure}[h]
\vskip.1cm
\begin{center}
\tiny{
\begin{tikzpicture}[scale=.725]%[domain=-2.8:2.8,scale=1]
%%%%%%%%%%%%%%%%%%%%%%%%%%%%%%%%%%%%%%%%
\draw (0,-4) node {$\boldsymbol{0}$};
%%%%%%%%%%%%%%%%%%%%%%%%%%%%%%%%%%%%%%%%
\draw (0,0) node {$
\left(
\begin{smallmatrix}
  0 \\
  0
\end{smallmatrix}
\right)$};
\draw[>=latex,->,thick] (0,-.3) -- (0,-3.7);
\draw (3,0) node {$
\left(
\begin{smallmatrix}
  1 \\
  1
\end{smallmatrix}
\right)
$};
\draw[>=latex,->,thick] (3,-.3) -- (0.2,-3.7);
\draw (6,0) node {$
\left(
\begin{smallmatrix}
  2 \\
  2
\end{smallmatrix}
\right)$};
\draw[>=latex,->,thick] (6,-.3) -- (0.3,-3.75);
\draw (9,0) node {$
\left(
\begin{smallmatrix}
  3 \\
  3
\end{smallmatrix}
\right)$};
\draw[>=latex,->,thick] (9,-.3) -- (0.45,-3.8);
\draw (12,0) node {$\boldsymbol{\cdots}$};
\draw (8.,-1.5) node {$\boldsymbol{\cdots}$};
\draw[>=latex,->,thick] (14.3,-.3) -- (0.6,-3.85);
\draw (15,0) node {$
\left(
\begin{smallmatrix}
  i_0-1 \\
  i_0-1
\end{smallmatrix}
\right)$};
\draw (18,0) node {$
\left(
\begin{smallmatrix}
  i_0 \\
  i_0
\end{smallmatrix}
\right)$};
\draw[>=latex,->,thick] (17.5,-.3) -- (0.8,-3.9);
\draw (21,0) node {$
\left(
\begin{smallmatrix}
  i_0+1 \\
  i_0+1
\end{smallmatrix}
\right)$};
\draw[>=latex,->,thick] (20.3,-.3) -- (0.99,-3.97);
\draw (24,0) node {$\boldsymbol{\cdots}$};
\draw (18.,-1.5) node {$\boldsymbol{\cdots}$};
%%%%%%%%%%%%%%%%%%%%%%%%%%%%%%%%%%%%%%%%
\draw (0,2) node {$
\left(
\begin{smallmatrix}
  0 & 1 \\
  0 & 1
\end{smallmatrix}
\right)$};
\draw[>=latex,->,thick] (0,1.7) -- (0,0.3);
\draw[>=latex,->,thick] (0.3,1.7) -- (2.7,0.3);
\draw (3,2) node {$
\left(
\begin{smallmatrix}
  1 & 2 \\
  1 & 2
\end{smallmatrix}
\right)$};
\draw[>=latex,->,thick] (3,1.7) -- (3,0.3);
\draw[>=latex,->,thick] (3.3,1.7) -- (5.7,0.3);
\draw (6,2) node {$
\left(
\begin{smallmatrix}
  2 & 3 \\
  2 & 3
\end{smallmatrix}
\right)$};
\draw[>=latex,->,thick] (6,1.7) -- (6,0.3);
\draw[>=latex,->,thick] (6.3,1.7) -- (8.7,0.3);
\draw (9,2) node {$
\left(
\begin{smallmatrix}
  3 & 4 \\
  3 & 4
\end{smallmatrix}
\right)$};
\draw[>=latex,->,thick] (9,1.7) -- (9,0.3);

\draw[>=latex,->,thick] (9.3,1.7) -- (10.6,0.9);
\draw (12,2) node {$\boldsymbol{\cdots}$};
\draw[>=latex,->,thick] (13.3,1.) -- (14.6,0.32);
\draw (15,2) node {$
\left(
\begin{smallmatrix}
  i_0{-}1 & i_0 \\
  i_0{-}1 & i_0
\end{smallmatrix}
\right)$};
\draw[>=latex,->,thick] (15.3,1.6) -- (17.7,0.3);
\draw[>=latex,->,thick] (15,1.6) -- (15,0.3);
\draw (18,2) node {$
\left(
\begin{smallmatrix}
  i_0 & i_0{+}1 \\
  i_0 & i_0{+}1
\end{smallmatrix}
\right)$};
\draw[>=latex,->,thick] (18,1.6) -- (18,0.3);
\draw[>=latex,->,thick] (18.3,1.6) -- (20.6,0.3);
\draw (21,2) node {$
\left(
\begin{smallmatrix}
  i_0{+}1 & i_0{+}2 \\
  i_0{+}1 & i_0{+}2
\end{smallmatrix}
\right)$};
\draw[>=latex,->,thick] (21,1.6) -- (21,0.3);
\draw[>=latex,->,thick] (21.3,1.6) -- (22.6,0.9);
\draw (24,2) node {$\boldsymbol{\cdots}$};
%%%%%%%%%%%%%%%%%%%%%%%%%%%%%%%%%%%%%%%%%
\draw (0,4) node {$
\left(
\begin{smallmatrix}
  0 & 1 & 2\\
  0 & 1 & 2
\end{smallmatrix}
\right)$};
\draw[>=latex,->,thick] (0,3.7) -- (0,2.3);
\draw[>=latex,->,thick] (0.3,3.7) -- (2.7,2.3);
\draw (3,4) node {$
\left(
\begin{smallmatrix}
  1 & 2 & 3\\
  1 & 2 & 3
\end{smallmatrix}
\right)$};
\draw[>=latex,->,thick] (3,3.7) -- (3,2.3);
\draw[>=latex,->,thick] (3.3,3.7) -- (5.7,2.3);
\draw (6,4) node {$
\left(
\begin{smallmatrix}
  2 & 3 & 4\\
  2 & 3 & 4
\end{smallmatrix}
\right)$};
\draw[>=latex,->,thick] (6,3.7) -- (6,2.3);
\draw[>=latex,->,thick] (6.3,3.7) -- (8.7,2.3);
\draw (9,4) node {$
\left(
\begin{smallmatrix}
  3 & 4 & 5\\
  3 & 4 & 5
\end{smallmatrix}
\right)$};
\draw[>=latex,->,thick] (9,3.7) -- (9,2.3);
\draw[>=latex,->,thick] (9.3,3.7) -- (10.6,2.9);
\draw (12,4) node {$\boldsymbol{\cdots}$};
\draw[>=latex,->,thick] (13.3,3.) -- (14.6,2.3);
\draw (15,4) node {$
\left(\!
\begin{smallmatrix}
  i_0{-}1 & i_0 & i_0{+}1\\
  i_0{-}1 & i_0 & i_0{+}1
\end{smallmatrix}
\!\right)$};
\draw[>=latex,->,thick] (15,3.6) -- (15,2.3);
\draw[>=latex,->,thick] (15.3,3.6) -- (17.7,2.3);
\draw (18,4) node {$
\left(\!
\begin{smallmatrix}
  i_0 & i_0{+}1 & i_0{+}2\\
  i_0 & i_0{+}1 & i_0{+}2
\end{smallmatrix}
\!\right)$};
\draw[>=latex,->,thick] (18,3.6) -- (18,2.3);
\draw[>=latex,->,thick] (18.3,3.6) -- (20.6,2.3);
\draw (21,4) node {$
\left(\!
\begin{smallmatrix}
  i_0{+}1 & i_0{+}2 & i_0{+}3\\
  i_0{+}1 & i_0{+}2 & i_0{+}3
\end{smallmatrix}
\!\right)$};
\draw[>=latex,->,thick] (21,3.6) -- (21,2.3);
\draw[>=latex,->,thick] (21.3,3.6) -- (22.6,2.9);
\draw (24,4) node {$\boldsymbol{\cdots}$};
%%%%%%%%%%%%%%%%%%%%%%%%%%%%%%%%%%%%%%%%%
\draw (0,6) node {$
\left(
\begin{smallmatrix}
  0 & 1 & 2 & 3\\
  0 & 1 & 2 & 3
\end{smallmatrix}
\right)$};
\draw[>=latex,->,thick] (0,5.7) -- (0,4.3);
\draw[>=latex,->,thick] (0.3,5.7) -- (2.7,4.3);
\draw (3,6) node {$
\left(
\begin{smallmatrix}
  1 & 2 & 3 & 4\\
  1 & 2 & 3 & 4
\end{smallmatrix}
\right)$};
\draw[>=latex,->,thick] (3,5.7) -- (3,4.3);
\draw[>=latex,->,thick] (3.3,5.7) -- (5.7,4.3);
\draw (6,6) node {$
\left(
\begin{smallmatrix}
  2 & 3 & 4 & 5\\
  2 & 3 & 4 & 5
\end{smallmatrix}
\right)$};
\draw[>=latex,->,thick] (6,5.7) -- (6,4.3);
\draw[>=latex,->,thick] (6.3,5.7) -- (8.7,4.3);
\draw (9,6) node {$
\left(
\begin{smallmatrix}
  3 & 4 & 5 & 6\\
  3 & 4 & 5 & 6
\end{smallmatrix}
\right)$};
\draw[>=latex,->,thick] (9,5.7) -- (9,4.3);
\draw[>=latex,->,thick] (9.3,5.7) -- (10.6,4.9);
\draw (12,6) node {$\boldsymbol{\cdots}$};
\draw[>=latex,->,thick] (13.3,5.) -- (14.6,4.3);
\draw (15,6) node {$
\left(\!
\begin{smallmatrix}
  i_0{-}1 & \cdots & i_0{+}2\\
  i_0{-}1 & \cdots & i_0{+}2
\end{smallmatrix}
\!\right)$};
\draw[>=latex,->,thick] (15,5.6) -- (15,4.3);
\draw[>=latex,->,thick] (15.3,5.6) -- (17.7,4.3);
\draw (18,6) node {$
\left(\!
\begin{smallmatrix}
  i_0 & \cdots & i_0{+}3\\
  i_0 & \cdots & i_0{+}3
\end{smallmatrix}
\!\right)$};
\draw[>=latex,->,thick] (18,5.6) -- (18,4.3);
\draw[>=latex,->,thick] (18.3,5.6) -- (20.6,4.3);
\draw (21,6) node {$
\left(\!
\begin{smallmatrix}
  i_0{+}1 & \cdots & i_0{+}4\\
  i_0{+}1 & \cdots & i_0{+}4
\end{smallmatrix}
\!\right)$};
\draw[>=latex,->,thick] (21,5.6) -- (21,4.3);
\draw[>=latex,->,thick] (21.3,5.6) -- (22.6,4.9);
\draw (24,6) node {$\boldsymbol{\cdots}$};
%%%%%%%%%%%%%%%%%%%%%%%%%%%%%%%%%%%%%%%%%
\draw (0,8) node {$\boldsymbol{\cdots}$};
\draw[>=latex,->,thick] (0,7.3) -- (0,6.3);
\draw[>=latex,->,thick] (1.,7.3) -- (2.7,6.3);
\draw (3,8) node {$\boldsymbol{\cdots}$};
\draw[>=latex,->,thick] (3,7.3) -- (3,6.3);
\draw[>=latex,->,thick] (4.,7.3) -- (5.7,6.3);
\draw (6,8) node {$\boldsymbol{\cdots}$};
\draw[>=latex,->,thick] (6,7.3) -- (6,6.3);
\draw[>=latex,->,thick] (7.,7.3) -- (8.7,6.3);
\draw (9,8) node {$\boldsymbol{\cdots}$};
\draw[>=latex,->,thick] (9,7.3) -- (9,6.3);
\draw[>=latex,->,thick] (16.,7.3) -- (17.7,6.3);
\draw (12,8) node {$\boldsymbol{\cdots}$};
\draw[>=latex,->,thick] (15,7.3) -- (15,6.3);
\draw[>=latex,->,thick] (10.,7.3) -- (10.7,6.9);
\draw (15,8) node {$\boldsymbol{\cdots}$};
\draw[>=latex,->,thick] (13.3,7.) -- (14.7,6.3);
\draw (18,8) node {$\boldsymbol{\cdots}$};
\draw[>=latex,->,thick] (18,7.3) -- (18,6.3);
\draw[>=latex,->,thick] (19.,7.3) -- (20.6,6.3);
\draw (21,8) node {$\boldsymbol{\cdots}$};
\draw[>=latex,->,thick] (21,7.3) -- (21,6.3);
\draw[>=latex,->,thick] (22.,7.3) -- (22.7,6.9);
\draw (24,8) node {$\boldsymbol{\cdots}$};
%%%%%%%%%%%%%%%%%%%%%%%%%%%%%%%%%%%%%%%%%
\draw (0,10) node {$
\left(\!
\begin{smallmatrix}
  0 & 1 & \cdots & n{-}2\\
  0 & 1 & \cdots & n{-}2
\end{smallmatrix}
\!\right)$};
\draw[>=latex,->,thick] (0,9.7) -- (0,8.7);
\draw[>=latex,->,thick] (0.3,9.7) -- (2.0,8.7);
\draw (3,10) node {$
\left(\!
\begin{smallmatrix}
  1 & 2 & \cdots & n{-}1\\
  1 & 2 & \cdots & n{-}1
\end{smallmatrix}
\!\right)$};
\draw[>=latex,->,thick] (3,9.7) -- (3,8.7);
\draw[>=latex,->,thick] (3.3,9.7) -- (5.0,8.7);
\draw (6,10) node {$
\left(\!
\begin{smallmatrix}
  2 & 3 & \cdots & n\\
  2 & 3 & \cdots & n
\end{smallmatrix}
\!\right)$};
\draw[>=latex,->,thick] (6,9.7) -- (6,8.7);
\draw[>=latex,->,thick] (6.3,9.7) -- (8.0,8.7);
\draw (9,10) node {$
\left(\!
\begin{smallmatrix}
  3 & 4 & \cdots & n{+}1\\
  3 & 4 & \cdots & n{+}1
\end{smallmatrix}
\!\right)$};
\draw[>=latex,->,thick] (9,9.7) -- (9,8.7);
\draw[>=latex,->,thick] (9.3,9.7) -- (10.6,8.9);

\draw (12,10) node {$\boldsymbol{\cdots}$};
\draw (15,10) node {$
\left(\!
\begin{smallmatrix}
  i_0{-}1 & \cdots & i_0{+}n{-}3\\
  i_0{-}1 & \cdots & i_0{+}n{-}3
\end{smallmatrix}
\!\right)$};
\draw[>=latex,->,thick] (15,9.6) -- (15,8.7);
\draw[>=latex,->,thick] (15.3,9.6) -- (16.9,8.8);

\draw[>=latex,->,thick] (13.3,9.4) -- (14.2,8.8);
\draw (18,10) node {$
\left(\!
\begin{smallmatrix}
  i_0 & \cdots & i_0{+}n{-}2\\
  i_0 & \cdots & i_0{+}n{-}2
\end{smallmatrix}
\!\right)$};
\draw[>=latex,->,thick] (18,9.6) -- (18,8.7);
\draw[>=latex,->,thick] (18.3,9.6) -- (19.6,8.9);
\draw (21,10) node {$
\left(\!
\begin{smallmatrix}
  i_0{+}1 & \cdots & i_0{+}n{-}1\\
  i_0{+}1 & \cdots & i_0{+}n{-}1
\end{smallmatrix}
\!\right)$};
\draw[>=latex,->,thick] (21,9.6) -- (21,8.7);
\draw[>=latex,->,thick] (21.3,9.6) -- (22.6,8.9);
\draw (24,10) node {$\boldsymbol{\cdots}$};
%%%%%%%%%%%%%%%%%%%%%%%%%%%%%%%%%%%%%%%%%
\draw (0,12) node {$
\left(\!
\begin{smallmatrix}
  0 & 1 & \cdots & n{-}1\\
  0 & 1 & \cdots & n{-}1
\end{smallmatrix}
\!\right)$};
\draw[>=latex,->,thick] (0,11.7) -- (0,10.3);
\draw[>=latex,->,thick] (0.3,11.7) -- (2.7,10.3);
\draw (3,12) node {$
\left(\!
\begin{smallmatrix}
  1 & 2 & \cdots & n\\
  1 & 2 & \cdots & n
\end{smallmatrix}
\!\right)$};
\draw[>=latex,->,thick] (3,11.7) -- (3,10.3);
\draw[>=latex,->,thick] (3.3,11.7) -- (5.7,10.3);
\draw (6,12) node {$
\left(\!
\begin{smallmatrix}
  2 & 3 & \cdots & n{+}1\\
  2 & 3 & \cdots & n{+}1
\end{smallmatrix}
\!\right)$};
\draw[>=latex,->,thick] (6,11.7) -- (6,10.3);
\draw[>=latex,->,thick] (6.3,11.7) -- (8.7,10.3);
\draw (9,12) node {$
\left(\!
\begin{smallmatrix}
  3 & 4 & \cdots & n{+}2\\
  3 & 4 & \cdots & n{+}2
\end{smallmatrix}
\!\right)$};
\draw[>=latex,->,thick] (9,11.7) -- (9,10.3);
\draw[>=latex,->,thick] (9.3,11.7) -- (10.6,10.9);
\draw (12,12) node {$\boldsymbol{\cdots}$};
\draw[>=latex,->,thick] (13.3,11.) -- (14.6,10.3);
\draw (15,12) node {$
\left(\!
\begin{smallmatrix}
  i_0{-}1 & \cdots & i_0{+}n{-}2\\
  i_0{-}1 & \cdots & i_0{+}n{-}2
\end{smallmatrix}
\!\right)$};
\draw[>=latex,->,thick] (15,11.6) -- (15,10.3);
\draw[>=latex,->,thick] (15.3,11.6) -- (17.7,10.3);
\draw (18,12) node {$
\left(\!
\begin{smallmatrix}
  i_0 & \cdots & i_0{+}n{-}1\\
  i_0 & \cdots & i_0{+}n{-}1
\end{smallmatrix}
\!\right)$};
\draw[>=latex,->,thick] (18,11.6) -- (18,10.3);
\draw[>=latex,->,thick] (18.3,11.6) -- (20.6,10.3);
\draw (21,12) node {$
\left(\!
\begin{smallmatrix}
  i_0{+}1 & \cdots & i_0{+}n\\
  i_0{+}1 & \cdots & i_0{+}n
\end{smallmatrix}
\!\right)$};
\draw[>=latex,->,thick] (21,11.6) -- (21,10.3);
\draw[>=latex,->,thick] (21.3,11.6) -- (22.6,10.9);
\draw (24,12) node {$\boldsymbol{\cdots}$};
%%%%%%%%%%%%%%%%%%%%%%%%%%%%%%%%%%%%%%%%%
\end{tikzpicture}
}
\end{center}
\vskip.1cm
\caption{
%The natural partial order on the band of the semigroup $\mathscr{I}_{\omega}^{n}$$(\overrightarrow{\operatorname{conv}})$
}\label{fig-2.1}
\end{figure}

For any $i_0\in \omega$ we define the endomorphism $\frak{e}_{i_0}\colon \mathscr{I}_{\omega}^{n}(\overrightarrow{\operatorname{conv}})\to \mathscr{I}_{\omega}^{n}(\overrightarrow{\operatorname{conv}})$ in the following way
\begin{equation*}
(\boldsymbol{0})\frak{e}_{i_0}{=}\boldsymbol{0}, \;
  \left(
\begin{smallmatrix}
  i \\
  j
\end{smallmatrix}
\right)\frak{e}_{i_0}{=}
\left(
\begin{smallmatrix}
  i+i_0 \\
  j+i_0
\end{smallmatrix}
\right)\!, \;
  \left(
\begin{smallmatrix}
  i & i+1\\
  j & j+1
\end{smallmatrix}
\right)\frak{e}_{i_0}{=}
\left(
\begin{smallmatrix}
  i+i_0 & i+1+i_0\\
  j+i_0 & j+1+i_0
\end{smallmatrix}
\right)\!, \; \ldots, \;
  \left(
\begin{smallmatrix}
  i & i+1 & \cdots & i+n-1\\
  j & j+1 & \cdots & j+n-1
\end{smallmatrix}
\right)\frak{e}_{i_0}{=}
\left(
\begin{smallmatrix}
  i+i_0 & i+1+i_0 & \cdots & i+n-1+i_0\\
  j+i_0 & j+1+i_0 & \cdots & j+n-1+i_0
\end{smallmatrix}
\right)\!.
\end{equation*}

\begin{theorem}\label{theorem-2.6}
Let $n$ be any positive integer $\geqslant2$. For every injective endomorphism $\frak{a}\colon \mathscr{I}_{\omega}^{n}(\overrightarrow{\operatorname{conv}})\to \mathscr{I}_{\omega}^{n}(\overrightarrow{\operatorname{conv}})$ there exists $i_0\in \omega$ such that  $\frak{a}=\frak{e}_{i_0}$.
\end{theorem}

\begin{proof}
By Lemma~\ref{lemma-2.3} we get that $(\boldsymbol{0})\frak{a}=\boldsymbol{0}$.

It is obvious that
\begin{equation*}
\mathfrak{M}=\left\{
  \left(
\begin{smallmatrix}
  i & i+1 & \cdots & i+n-1\\
  i & i+1 & \cdots & i+n-1
\end{smallmatrix}
\right)\colon i\in\omega\right\}
\end{equation*}
is the set of all maximal idempotents of $E(\mathscr{I}_{\omega}^{n}(\overrightarrow{\operatorname{conv}}))$, and moreover every maximal chain in the semilattice $E(\mathscr{I}_{\omega}^{n}(\overrightarrow{\operatorname{conv}}))$ contains $n+1$ idempotents. Hence
\begin{equation*}
  L_0=\left\{\boldsymbol{0}, \left(
\begin{smallmatrix}
  0 \\
  0
\end{smallmatrix}
\right),
\left(
\begin{smallmatrix}
  0 & 1 \\
  0 & 1
\end{smallmatrix}
\right), \ldots ,
\left(
\begin{smallmatrix}
  0 & 1 & \cdots & n-1\\
  0 & 1 & \cdots & n-1
\end{smallmatrix}
\right)
\right\}
\end{equation*}
and
\begin{equation*}
  L_1=\left\{\boldsymbol{0}, \left(
\begin{smallmatrix}
  1 \\
  1
\end{smallmatrix}
\right),
\left(
\begin{smallmatrix}
  1 & 2 \\
  1 & 2
\end{smallmatrix}
\right), \ldots ,
\left(
\begin{smallmatrix}
  1 & 2 & \cdots & n\\
  1 & 2 & \cdots & n
\end{smallmatrix}
\right)
\right\}
\end{equation*}
are maximal chains in $E(\mathscr{I}_{\omega}^{n}(\overrightarrow{\operatorname{conv}}))$. Since $\frak{a}$ is an injective endomorphism of the semigroup $\mathscr{I}_{\omega}^{n}(\overrightarrow{\operatorname{conv}})$, Proposition~1.14.21(6) of \cite{Lawson-1998} implies that the images $(L_0)\frak{a}$ and $(L_1)\frak{a}$ are maximal chains in $E(\mathscr{I}_{\omega}^{n}(\overrightarrow{\operatorname{conv}}))$.

Put
$
\left(
\begin{smallmatrix}
  0 & 1 & \cdots & n-1\\
  0 & 1 & \cdots & n-1
\end{smallmatrix}
\right)\frak{a}=
\left(
\begin{smallmatrix}
  i_0 & i_0+1 & \cdots & i_0+n-1\\
  i_0 & i_0+1 & \cdots & i_0+n-1
\end{smallmatrix}
\right)\in\mathfrak{M}
$. Since
$\operatorname{rank}
\left(
\begin{smallmatrix}
  1 & \cdots & n-1\\
  1 & \cdots & n-1
\end{smallmatrix}
\right)=n-1
$, $
\left(
\begin{smallmatrix}
  1 & \cdots & n-1\\
  1 & \cdots & n-1
\end{smallmatrix}
\right)\preccurlyeq
\left(
\begin{smallmatrix}
  0 & 1 & \cdots & n-1\\
  0 & 1 & \cdots & n-1
\end{smallmatrix}
\right)$ and $
\left(
\begin{smallmatrix}
  1 & \cdots & n-1\\
  1 & \cdots & n-1
\end{smallmatrix}
\right)\preccurlyeq
\left(
\begin{smallmatrix}
  1 & 2 & \cdots & n\\
  1 & 2 & \cdots & n
\end{smallmatrix}
\right)$, the definition of the natural partial order on the semilattice $E(\mathscr{I}_{\omega}^{n}(\overrightarrow{\operatorname{conv}}))$ and Proposition~1.14.21(6) of \cite{Lawson-1998} imply that either $
\left(
\begin{smallmatrix}
  1 & \cdots & n-1\\
  1 & \cdots & n-1
\end{smallmatrix}
\right)\frak{a}=
\left(
\begin{smallmatrix}
  i_0 & i_0+1 & \cdots & i_0+n-2\\
  i_0 & i_0+1 & \cdots & i_0+n-2
\end{smallmatrix}
\right)
$ or
$
\left(
\begin{smallmatrix}
  1 & \cdots & n-1\\
  1 & \cdots & n-1
\end{smallmatrix}
\right)\frak{a}=
\left(
\begin{smallmatrix}
  i_0+1 & i_0+2 & \cdots & i_0+n-1\\
  i_0+1 & i_0+2 & \cdots & i_0+n-1
\end{smallmatrix}
\right)
$.

Suppose that $
\left(
\begin{smallmatrix}
  1 & \cdots & n-1\\
  1 & \cdots & n-1
\end{smallmatrix}
\right)\frak{a}=
\left(
\begin{smallmatrix}
  i_0 & i_0+1 & \cdots & i_0+n-2\\
  i_0 & i_0+1 & \cdots & i_0+n-2
\end{smallmatrix}
\right)
$. Since $
\left(
\begin{smallmatrix}
  1 & 2 & \cdots & n\\
  1 & 2 & \cdots & n
\end{smallmatrix}
\right)\in\mathfrak{M}
$ we have that $
\left(
\begin{smallmatrix}
  1 & 2 & \cdots & n\\
  1 & 2 & \cdots & n
\end{smallmatrix}
\right)\frak{a}\in\mathfrak{M}
$, and the definition of the natural partial order on the semilattice $E(\mathscr{I}_{\omega}^{n}(\overrightarrow{\operatorname{conv}}))$ and Proposition~1.14.21(6) of \cite{Lawson-1998} imply that $
\left(
\begin{smallmatrix}
  1 & 2 & \cdots & n\\
  1 & 2 & \cdots & n
\end{smallmatrix}
\right)\frak{a}=
\left(
\begin{smallmatrix}
  i_0-1 & i_0 & \cdots & i_0+n-2\\
  i_0-1 & i_0 & \cdots & i_0+n-2
\end{smallmatrix}
\right)
$. Again, by the definition of the natural partial order on the semilattice $E(\mathscr{I}_{\omega}^{n}(\overrightarrow{\operatorname{conv}}))$ and Proposition~1.14.21(6) of \cite{Lawson-1998} we obtain that $
\left(
\begin{smallmatrix}
  2 & \cdots & n\\
  2 & \cdots & n
\end{smallmatrix}
\right)\frak{a}=
\left(
\begin{smallmatrix}
  i_0-1 & i_0 & \cdots & i_0+n-3\\
  i_0-1 & i_0 & \cdots & i_0+n-3
\end{smallmatrix}
\right)
$ because $\operatorname{rank}
\left(
\begin{smallmatrix}
  2 & \cdots & n\\
  2 & \cdots & n
\end{smallmatrix}
\right)=n-1
$ and $
\left(
\begin{smallmatrix}
  2 & \cdots & n\\
  2 & \cdots & n
\end{smallmatrix}
\right)\preccurlyeq
\left(
\begin{smallmatrix}
  1 & 2 & \cdots & n\\
  1 & 2 & \cdots & n
\end{smallmatrix}
\right)$. Since  $
\left(
\begin{smallmatrix}
  2 & \cdots & n\\
  2 & \cdots & n
\end{smallmatrix}
\right)\preccurlyeq
\left(
\begin{smallmatrix}
  2 & 3 & \cdots & n+1\\
  2 & 3 & \cdots & n+1
\end{smallmatrix}
\right)$ the above arguments imply that $
\left(
\begin{smallmatrix}
  2 & 3 & \cdots & n+1\\
  2 & 3 & \cdots & n+1
\end{smallmatrix}
\right)\frak{a}=
\left(
\begin{smallmatrix}
  i_0-2 & i_0-1 & \cdots & i_0+n-3\\
  i_0-2 & i_0-1 & \cdots & i_0+n-3
\end{smallmatrix}
\right)
$ and $
\left(
\begin{smallmatrix}
  3 & \cdots & n+1\\
  3 & \cdots & n+1
\end{smallmatrix}
\right)\frak{a}=
\left(
\begin{smallmatrix}
  i_0-2 & i_0-1 & \cdots & i_0+n-4\\
  i_0-2 & i_0-1 & \cdots & i_0+n-4
\end{smallmatrix}
\right)
$. Next, we extend the above procedure step-by-step using the definition of the natural partial order on the semilattice $E(\mathscr{I}_{\omega}^{n}(\overrightarrow{\operatorname{conv}}))$ and Proposition~1.14.21(6) of \cite{Lawson-1998} we get that $
\left(
\begin{smallmatrix}
  i_0+1 & \cdots & i_0+n-1\\
  i_0+1 & \cdots & i_0+n-1
\end{smallmatrix}
\right)\frak{a}=
\left(
\begin{smallmatrix}
  0 & 1 & \cdots & n-2\\
  0 & 1 & \cdots & n-2
\end{smallmatrix}
\right)
$ and $
\left(
\begin{smallmatrix}
  i_0 & i_0+1 & \cdots & i_0+n-1\\
  i_0 & i_0+1 & \cdots & i_0+n-1
\end{smallmatrix}
\right)\frak{a}=
\left(
\begin{smallmatrix}
  0 & 1 & \cdots & n-1\\
  0 & 1 & \cdots & n-1
\end{smallmatrix}
\right)
$. Since $\frak{a}$ is an injective endomorphism of the semigroup $\mathscr{I}_{\omega}^{n}(\overrightarrow{\operatorname{conv}})$, $
\left(
\begin{smallmatrix}
  i_0+1 & \cdots & i_0+n-1\\
  i_0+1 & \cdots & i_0+n-1
\end{smallmatrix}
\right)\preccurlyeq
\left(
\begin{smallmatrix}
  i_0 & i_0+1 & \cdots & i_0+n-1\\
  i_0 & i_0+1 & \cdots & i_0+n-1
\end{smallmatrix}
\right)$ and $
\left(
\begin{smallmatrix}
  i_0+1 & \cdots & i_0+n-1\\
  i_0+1 & \cdots & i_0+n-1
\end{smallmatrix}
\right)\preccurlyeq
\left(
\begin{smallmatrix}
  i_0+1 & i_0+2 & \cdots & i_0+n\\
  i_0+1 & i_0+2 & \cdots & i_0+n
\end{smallmatrix}
\right)$ in $E(\mathscr{I}_{\omega}^{n}(\overrightarrow{\operatorname{conv}}))$, we conclude that Proposition~1.14.21(6) of \cite{Lawson-1998} implies that $
\left(
\begin{smallmatrix}
  i_0+1 & \cdots & i_0+n-1\\
  i_0+1 & \cdots & i_0+n-1
\end{smallmatrix}
\right)\frak{a}\preccurlyeq
\left(
\begin{smallmatrix}
  i_0 & i_0+1 & \cdots & i_0+n-1\\
  i_0 & i_0+1 & \cdots & i_0+n-1
\end{smallmatrix}
\right)\frak{a}$ and \linebreak $
\left(
\begin{smallmatrix}
  i_0+1 & \cdots & i_0+n-1\\
  i_0+1 & \cdots & i_0+n-1
\end{smallmatrix}
\right)\frak{a}\preccurlyeq
\left(
\begin{smallmatrix}
  i_0+1 & i_0+2 & \cdots & i_0+n\\
  i_0+1 & i_0+2 & \cdots & i_0+n
\end{smallmatrix}
\right)\frak{a}$. But  $
\left(
\begin{smallmatrix}
  i_0 & i_0+1 & \cdots & i_0+n-1\\
  i_0 & i_0+1 & \cdots & i_0+n-1
\end{smallmatrix}
\right)\frak{a}=
\left(
\begin{smallmatrix}
  0 & 1 & \cdots & n-1\\
  0 & 1 & \cdots & n-1
\end{smallmatrix}
\right)
$ is the unique idempotent of the semilattice $E(\mathscr{I}_{\omega}^{n}(\overrightarrow{\operatorname{conv}}))$ which is greater than $
\left(
\begin{smallmatrix}
  i_0+1 & \cdots & i_0+n-1\\
  i_0+1 & \cdots & i_0+n-1
\end{smallmatrix}
\right)\frak{a}=
\left(
\begin{smallmatrix}
  0 & 1 & \cdots & n-2\\
  0 & 1 & \cdots & n-2
\end{smallmatrix}
\right)
$. The obtained contradiction implies that $
\left(
\begin{smallmatrix}
  1 & \cdots & n-1\\
  1 & \cdots & n-1
\end{smallmatrix}
\right)\frak{a}\neq
\left(
\begin{smallmatrix}
  i_0 & i_0+1 & \cdots & i_0+n-2\\
  i_0 & i_0+1 & \cdots & i_0+n-2
\end{smallmatrix}
\right)
$ and hence we get that  $
\left(
\begin{smallmatrix}
  1 & \cdots & n-1\\
  1 & \cdots & n-1
\end{smallmatrix}
\right)\frak{a}=
\left(
\begin{smallmatrix}
  i_0+1 & i_0+2 & \cdots & i_0+n-1\\
  i_0+1 & i_0+2 & \cdots & i_0+n-1
\end{smallmatrix}
\right)
$.

The inequality $
\left(
\begin{smallmatrix}
  0 & 1 & \cdots & n-2\\
  0 & 1 & \cdots & n-2
\end{smallmatrix}
\right)\preccurlyeq
\left(
\begin{smallmatrix}
  0 & 1 & \cdots & n-1\\
  0 & 1 & \cdots & n-1
\end{smallmatrix}
\right)
$ implies that $
\left(
\begin{smallmatrix}
  0 & 1 & \cdots & n-2\\
  0 & 1 & \cdots & n-2
\end{smallmatrix}
\right)\frak{a}\preccurlyeq
\left(
\begin{smallmatrix}
  0 & 1 & \cdots & n-1\\
  0 & 1 & \cdots & n-1
\end{smallmatrix}
\right)\frak{a}
$, and hence the definition of the natural partial order on  $E(\mathscr{I}_{\omega}^{n}(\overrightarrow{\operatorname{conv}}))$, injectivity of $\frak{a}$, Proposition~1.14.21(6) of \cite{Lawson-1998} and the equality $
\left(
\begin{smallmatrix}
  1 & \cdots & n-1\\
  1 & \cdots & n-1
\end{smallmatrix}
\right)\frak{a}=
\left(
\begin{smallmatrix}
  i_0+1 & i_0+2 & \cdots & i_0+n-1\\
  i_0+1 & i_0+2 & \cdots & i_0+n-1
\end{smallmatrix}
\right)
$ imply that $
\left(
\begin{smallmatrix}
  0 & 1 & \cdots & n-2\\
  0 & 1 & \cdots & n-2
\end{smallmatrix}
\right)\frak{a}=
\left(
\begin{smallmatrix}
  i_0 & i_0+1 & \cdots & i_0+n-2\\
  i_0 & i_0+1 & \cdots & i_0+n-2
\end{smallmatrix}
\right)
$. Again, since $
\left(
\begin{smallmatrix}
  1 & 2 & \cdots & n-2\\
  1 & 2 & \cdots & n-2
\end{smallmatrix}
\right)\preccurlyeq
\left(
\begin{smallmatrix}
  1 & 2 & \cdots & n-1\\
  1 & 2 & \cdots & n-1
\end{smallmatrix}
\right)
$ and $
\left(
\begin{smallmatrix}
  1 & 2 & \cdots & n-2\\
  1 & 2 & \cdots & n-2
\end{smallmatrix}
\right)\preccurlyeq
\left(
\begin{smallmatrix}
  0 & 1 & \cdots & n-2\\
  0 & 1 & \cdots & n-2
\end{smallmatrix}
\right)
$ we obtain that $
\left(
\begin{smallmatrix}
  1 & 2 & \cdots & n-2\\
  1 & 2 & \cdots & n-2
\end{smallmatrix}
\right)\frak{a}\preccurlyeq
\left(
\begin{smallmatrix}
  1 & 2 & \cdots & n-1\\
  1 & 2 & \cdots & n-1
\end{smallmatrix}
\right)\frak{a}
$ and $
\left(
\begin{smallmatrix}
  1 & 2 & \cdots & n-2\\
  1 & 2 & \cdots & n-2
\end{smallmatrix}
\right)\frak{a}\preccurlyeq
\left(
\begin{smallmatrix}
  0 & 1 & \cdots & n-2\\
  0 & 1 & \cdots & n-2
\end{smallmatrix}
\right)\frak{a}
$. The above two inequalities  and the equalities \linebreak $
\left(
\begin{smallmatrix}
  1 & \cdots & n-1\\
  1 & \cdots & n-1
\end{smallmatrix}
\right)\frak{a}=
\left(
\begin{smallmatrix}
  i_0+1 & i_0+2 & \cdots & i_0+n-1\\
  i_0+1 & i_0+2 & \cdots & i_0+n-1
\end{smallmatrix}
\right)
$ and $
\left(
\begin{smallmatrix}
  0 & 1 & \cdots & n-2\\
  0 & 1 & \cdots & n-2
\end{smallmatrix}
\right)\frak{a}=
\left(
\begin{smallmatrix}
  i_0 & i_0+1 & \cdots & i_0+n-2\\
  i_0 & i_0+1 & \cdots & i_0+n-2
\end{smallmatrix}
\right)
$ imply that $
\left(
\begin{smallmatrix}
  1 & 2 & \cdots & n-2\\
  1 & 2 & \cdots & n-2
\end{smallmatrix}
\right)\frak{a}=
\left(
\begin{smallmatrix}
  i_0+1 & i_0+2 & \cdots & i_0+n-2\\
  i_0+1 & i_0+2 & \cdots & i_0+n-2
\end{smallmatrix}
\right)
$. Now, if we repeat the above procedure step-by-step we get the following equalities
\begin{align*}
  \left(
\begin{smallmatrix}
  0 & 1 & \cdots & n-1\\
  0 & 1 & \cdots & n-1
\end{smallmatrix}
\right)\frak{a}&=
\left(
\begin{smallmatrix}
  i_0 & i_0+1 & \cdots & i_0+n-1\\
  i_0 & i_0+1 & \cdots & i_0+n-1
\end{smallmatrix}
\right),  &
  \left(
\begin{smallmatrix}
  1 & 2 & \cdots & n\\
  1 & 2 & \cdots & n
\end{smallmatrix}
\right)\frak{a}&=
\left(
\begin{smallmatrix}
  i_0+1 & i_0+2 & \cdots & i_0+n\\
  i_0+1 & i_0+2 & \cdots & i_0+n
\end{smallmatrix}
\right),\\
  \left(
\begin{smallmatrix}
  0 & 1 & \cdots & n-2\\
  0 & 1 & \cdots & n-2
\end{smallmatrix}
\right)\frak{a}&=
\left(
\begin{smallmatrix}
  i_0 & i_0+1 & \cdots & i_0+n-2\\
  i_0 & i_0+1 & \cdots & i_0+n-2
\end{smallmatrix}
\right),  &
  \left(
\begin{smallmatrix}
  1 & 2 & \cdots & n-1\\
  1 & 2 & \cdots & n-1
\end{smallmatrix}
\right)\frak{a}&=
\left(
\begin{smallmatrix}
  i_0+1 & i_0+2 & \cdots & i_0+n-1\\
  i_0+1 & i_0+2 & \cdots & i_0+n-1
\end{smallmatrix}
\right),\\
  \cdots \qquad & \qquad  \qquad \cdots \qquad \qquad & \qquad  \qquad \cdots \qquad  & \qquad  \qquad \cdots \qquad \qquad \\
  \left(
\begin{smallmatrix}
  0 & 1 \\
  0 & 1
\end{smallmatrix}
\right)\frak{a}&=
\left(
\begin{smallmatrix}
  i_0 & i_0+1\\
  i_0 & i_0+1
\end{smallmatrix}
\right),  &
  \left(
\begin{smallmatrix}
  1 & 2\\
  1 & 2
\end{smallmatrix}
\right)\frak{a}&=
\left(
\begin{smallmatrix}
  i_0+1 & i_0+2\\
  i_0+1 & i_0+2
\end{smallmatrix}
\right),\\
  \left(
\begin{smallmatrix}
  0 \\
  0
\end{smallmatrix}
\right)\frak{a}&=
\left(
\begin{smallmatrix}
  i_0\\
  i_0
\end{smallmatrix}
\right),  &
  \left(
\begin{smallmatrix}
  1\\
  1
\end{smallmatrix}
\right)\frak{a}&=
\left(
\begin{smallmatrix}
  i_0+1\\
  i_0+1
\end{smallmatrix}
\right).
\end{align*}
Thus, we show that the  initial case of induction holds.

Next we shall prove that the induction step holds: if for positive integer the equalities
\begin{align*}
  \left(
\begin{smallmatrix}
  p & p+1 & \cdots & p+n-1\\
  p & p+1 & \cdots & p+n-1
\end{smallmatrix}
\right)\frak{a}&=
\left(
\begin{smallmatrix}
  p+i_0 & p+i_0+1 & \cdots & p+i_0+n-1\\
  p+i_0 & p+i_0+1 & \cdots & p+i_0+n-1
\end{smallmatrix}
\right),  \\
    \left(
\begin{smallmatrix}
  p & p+1 & \cdots & p+n-2\\
  p & p+1 & \cdots & p+n-2
\end{smallmatrix}
\right)\frak{a}&=
\left(
\begin{smallmatrix}
  p+i_0 & p+i_0+1 & \cdots & p+i_0+n-2\\
  p+i_0 & p+i_0+1 & \cdots & p+i_0+n-2
\end{smallmatrix}
\right),  \\
  \cdots \qquad & \qquad  \qquad \cdots \qquad \qquad  \\
  \left(
\begin{smallmatrix}
  p & p+1 \\
  p & p+1
\end{smallmatrix}
\right)\frak{a}&=
\left(
\begin{smallmatrix}
  p+i_0 & p+i_0+1\\
  p+i_0 & p+i_0+1
\end{smallmatrix}
\right),  \\
  \left(
\begin{smallmatrix}
  p \\
  p
\end{smallmatrix}
\right)\frak{a}&=
\left(
\begin{smallmatrix}
  p+i_0\\
  p+i_0
\end{smallmatrix}
\right)
\end{align*}
hold for $p=0,1,\ldots,k$, then they hold for $p=k+1$.

By the assumption of induction we have that $
\left(
\begin{smallmatrix}
  k & k+1 & \cdots & k+n-1\\
  k & k+1 & \cdots & k+n-1
\end{smallmatrix}
\right)\frak{a}=
\left(
\begin{smallmatrix}
  k+i_0 & k+i_0+1 & \cdots & k+i_0+n-1\\
  k+i_0 & k+i_0+1 & \cdots & k+i_0+n-1
\end{smallmatrix}
\right)
$ and \linebreak $
\left(
\begin{smallmatrix}
  k & k+1 & \cdots & k+n-2\\
  k & k+1 & \cdots & k+n-2
\end{smallmatrix}
\right)\frak{a}=
\left(
\begin{smallmatrix}
  k+i_0 & k+i_0+1 & \cdots & k+i_0+n-2\\
  k+i_0 & k+i_0+1 & \cdots & k+i_0+n-2
\end{smallmatrix}
\right)
$. Since the endomorphism $\frak{a}$ is injective, this, the inequalities $
\left(
\begin{smallmatrix}
  k & k+1 & \cdots & k+n-2\\
  k & k+1 & \cdots & k+n-2
\end{smallmatrix}
\right)\preccurlyeq
\left(
\begin{smallmatrix}
  k & k+1 & \cdots & k+n-1\\
  k & k+1 & \cdots & k+n-1
\end{smallmatrix}
\right)
$ and $
\left(
\begin{smallmatrix}
  k+1 & k+2 & \cdots & k+n-1\\
  k+1 & k+2 & \cdots & k+n-1
\end{smallmatrix}
\right)\preccurlyeq
\left(
\begin{smallmatrix}
  k & k+1 & \cdots & k+n-1\\
  k & k+1 & \cdots & k+n-1
\end{smallmatrix}
\right)
$, the definition of the natural partial order on the semilattice $E(\mathscr{I}_{\omega}^{n}(\overrightarrow{\operatorname{conv}}))$ and Proposition~1.14.21(6) of \cite{Lawson-1998} imply that $
\left(
\begin{smallmatrix}
  k+1 & k+2 & \cdots & k+n-1\\
  k+1 & k+2 & \cdots & k+n-1
\end{smallmatrix}
\right)\frak{a}=
\left(
\begin{smallmatrix}
  k+i_0+1 & k+i_0+2 & \cdots & k+i_0+n-1\\
  k+i_0+1 & k+i_0+2 & \cdots & k+i_0+n-1
\end{smallmatrix}
\right)
$. Again, since $
\left(
\begin{smallmatrix}
  k+1 & k+2 & \cdots & k+n\\
  k+1 & k+2 & \cdots & k+n
\end{smallmatrix}
\right)
$ is the unique idempotent of $E(\mathscr{I}_{\omega}^{n}(\overrightarrow{\operatorname{conv}}))$ which is greater than $
\left(
\begin{smallmatrix}
  k+1 & k+2 & \cdots & k+n-1\\
  k+1 & k+2 & \cdots & k+n-1
\end{smallmatrix}
\right)
$ and it is  distinct from the idempotent $
\left(
\begin{smallmatrix}
  k & k+1 & \cdots & k+n-1\\
  k & k+1 & \cdots & k+n-1
\end{smallmatrix}
\right)
$, the definition of the natural partial order on the semilattice $E(\mathscr{I}_{\omega}^{n}(\overrightarrow{\operatorname{conv}}))$ and Proposition~1.14.21(6) of \cite{Lawson-1998} imply that $
\left(
\begin{smallmatrix}
  k+1 & k+2 & \cdots & k+n\\
  k+1 & k+2 & \cdots & k+n
\end{smallmatrix}
\right)\frak{a}=
\left(
\begin{smallmatrix}
  k+i_0+1 & k+i_0+2 & \cdots & k+i_0+n\\
  k+i_0+1 & k+i_0+2 & \cdots & k+i_0+n
\end{smallmatrix}
\right)
$. Next, the equality $
\left(
\begin{smallmatrix}
  k+1 & k+2 & \cdots & k+n-1\\
  k+1 & k+2 & \cdots & k+n-1
\end{smallmatrix}
\right)\frak{a}=
\left(
\begin{smallmatrix}
  k+i_0+1 & k+i_0+2 & \cdots & k+i_0+n-1\\
  k+i_0+1 & k+i_0+2 & \cdots & k+i_0+n-1
\end{smallmatrix}
\right)
$ and the above presented argument imply that $
\left(
\begin{smallmatrix}
  k+1 & k+2 & \cdots & k+n-2\\
  k+1 & k+2 & \cdots & k+n-2
\end{smallmatrix}
\right)\frak{a}=
\left(
\begin{smallmatrix}
  k+i_0+1 & k+i_0+2 & \cdots & k+i_0+n-2\\
  k+i_0+1 & k+i_0+2 & \cdots & k+i_0+n-2
\end{smallmatrix}
\right)
$, and by the similar way step-by-step we obtain that the following equalities
\begin{align*}
  \left(
\begin{smallmatrix}
  k+1 & k+2 & \cdots & k+n\\
  k+1 & k+2 & \cdots & k+n
\end{smallmatrix}
\right)\frak{a}&=
\left(
\begin{smallmatrix}
  k+i_0+1 & k+i_0+2 & \cdots & k+i_0+n\\
  k+i_0+1 & k+i_0+2 & \cdots & k+i_0+n
\end{smallmatrix}
\right), \\
  \left(
\begin{smallmatrix}
  k+1 & k+2 & \cdots & k+n-1\\
  k+1 & k+2 & \cdots & k+n-1
\end{smallmatrix}
\right)\frak{a}&=
\left(
\begin{smallmatrix}
  k+i_0+1 & k+i_0+2 & \cdots & k+i_0+n-1\\
  k+i_0+1 & k+i_0+2 & \cdots & k+i_0+n-1
\end{smallmatrix}
\right), \\
\cdots \qquad & \qquad  \qquad \cdots \qquad \qquad  \\
  \left(
\begin{smallmatrix}
  k+1 & k+2 \\
  k+1 & k+2
\end{smallmatrix}
\right)\frak{a}&=
\left(
\begin{smallmatrix}
  k+i_0+1 & k+i_0+2 \\
  k+i_0+1 & k+i_0+2
\end{smallmatrix}
\right), \\
 \left(
\begin{smallmatrix}
  k+1 \\
  k+1
\end{smallmatrix}
\right)\frak{a}&=
\left(
\begin{smallmatrix}
  k+i_0+1 \\
  k+i_0+1
\end{smallmatrix}
\right)
\end{align*}
hold, and hence we proved the step of induction.

Fix an arbitrary non-idempotent element $\boldsymbol{x}=
\left(
\begin{smallmatrix}
  a & a+1 & \cdots & a+m \\
  b & b+1 & \cdots & b+m
\end{smallmatrix}
\right)
$ of the semigroup $\mathscr{I}_{\omega}^{n}(\overrightarrow{\operatorname{conv}})$, for some $a,b\in \omega$ and $m=0,1,\ldots,n-1$. Then $\boldsymbol{x}\boldsymbol{x}^{-1}=
\left(
\begin{smallmatrix}
  a & a+1 & \cdots & a+m \\
  a & a+1 & \cdots & a+m
\end{smallmatrix}
\right)
$ and $\boldsymbol{x}^{-1}\boldsymbol{x}=
\left(
\begin{smallmatrix}
  b & b+1 & \cdots & b+m \\
  b & b+1 & \cdots & b+m
\end{smallmatrix}
\right)
$, and hence by the previous part of the proof we have that
\begin{equation*}
  (\boldsymbol{x}\boldsymbol{x}^{-1})\mathfrak{a}=
  \left(
\begin{smallmatrix}
  i_0+a & i_0+a+1 & \cdots & i_0+a+m \\
  i_0+a & i_0+a+1 & \cdots & i_0+a+m
\end{smallmatrix}
\right)
\qquad \hbox{and} \qquad
  (\boldsymbol{x}^{-1}\boldsymbol{x})\mathfrak{a}=
  \left(
\begin{smallmatrix}
  i_0+b & i_0+b+1 & \cdots & i_0+b+m \\
  i_0+b & i_0+b+1 & \cdots & i_0+b+m
\end{smallmatrix}
\right).
\end{equation*}
Since $\mathscr{I}_{\omega}^{n}(\overrightarrow{\operatorname{conv}})$ is an inverse subsemigroup of the symmetric inverse monoid $\mathscr{I}_{\omega}$ over $\omega$, we conclude that
\begin{equation*}
\operatorname{dom}((\boldsymbol{x})\mathfrak{a})=\operatorname{dom}((\boldsymbol{x}\boldsymbol{x}^{-1})\mathfrak{a})=\left\{i_0+a, i_0+a+1, \ldots, i_0+a+m\right\}
\end{equation*}
and
\begin{equation*}
\operatorname{ran}((\boldsymbol{x})\mathfrak{a})=\operatorname{ran}((\boldsymbol{x}^{-1}\boldsymbol{x})\mathfrak{a})= \left\{i_0+b, i_0+b+1, \ldots, i_0+b+m\right\}.
\end{equation*}
Now, the definition of the semigroup $\mathscr{I}_{\omega}^{n}(\overrightarrow{\operatorname{conv}})$ implies that
\begin{equation*}
  (\boldsymbol{x})\mathfrak{a}=
  \left(
\begin{smallmatrix}
  i_0+a & i_0+a+1 & \cdots & i_0+a+m \\
  i_0+b & i_0+b+1 & \cdots & i_0+b+m
\end{smallmatrix}
\right).
\end{equation*}
By Corollary~\ref{corollary-2.2}, $\mathfrak{a}=\mathfrak{e}_{i_0}$ is an endomorphism of the semigroup $\mathscr{I}_{\omega}^{n}(\overrightarrow{\operatorname{conv}})$,
which completes the proof of the theorem.
\end{proof}

Lemma~\ref{lemma-2.5} and Theorem~\ref{theorem-2.6} imply

\begin{corollary}\label{corollary-2.7}
For any positive integer $n\geqslant2$ every automorphism of the semigroup $\mathscr{I}_{\omega}^{n}(\overrightarrow{\operatorname{conv}})$ is the identity map of $\mathscr{I}_{\omega}^{n}(\overrightarrow{\operatorname{conv}})$.
\end{corollary}

For any positive integer $n$ and any injective endomorphisms $\mathfrak{e}_{i_1}$ and $\mathfrak{e}_{i_2}$ of the semigroup $\mathscr{I}_{\omega}^{n}(\overrightarrow{\operatorname{conv}})$ simple calculations show that
\begin{equation*}
  \mathfrak{e}_{i_1}\circ\mathfrak{e}_{i_2}=\mathfrak{e}_{i_1+i_2}=\mathfrak{e}_{i_2}\circ\mathfrak{e}_{i_1}.
\end{equation*}

This and Theorem~\ref{theorem-2.6} imply

\begin{theorem}\label{theorem-2.8}
For any positive integer $n\geqslant2$ the semigroup of injective endomorphisms of the semigroup $\mathscr{I}_{\omega}^{n}(\overrightarrow{\operatorname{conv}})$ is  isomorphic to the  semigroup $(\omega,+)$. In particular the group of automorphisms of  $\mathscr{I}_{\omega}^{n}(\overrightarrow{\operatorname{conv}})$ is trivial.
\end{theorem}

Since by Theorem~3 of \cite{Gutik-Popadiuk=2022} for any $n\in\omega$ the semigroup $\boldsymbol{B}_{\omega}^{\mathscr{F}_n}$ is isomorphic to the semigroup $\mathscr{I}_\omega^{n+1}(\overrightarrow{\mathrm{conv}})$, Corollary~\ref{corollary-2.7} and Theorem~\ref{theorem-2.8} imply the following two corollaries.

\begin{corollary}\label{corollary-2.9}
For any positive integer $n$ every automorphism of the semigroup $\boldsymbol{B}_{\omega}^{\mathscr{F}_n}$ is the identity map of $\boldsymbol{B}_{\omega}^{\mathscr{F}_n}$.
\end{corollary}

\begin{corollary}\label{corollary-2.10}
For any positive integer $n$ the semigroup of injective endomorphisms of the semigroup $\boldsymbol{B}_{\omega}^{\mathscr{F}_n}$ is  isomorphic to the semigroup $(\omega,+)$. In particular the group of automorphisms of  $\boldsymbol{B}_{\omega}^{\mathscr{F}_n}$ is trivial.
\end{corollary}

\section{On  endomorphisms of the semigroup of $\lambda{\times}\lambda$-matrix units}

Let $\lambda$ be a non-zero cardinal and $\boldsymbol{0}\notin \lambda{\times}\lambda$. The set $\mathscr{B}_{\lambda}=\lambda{\times}\lambda\cup\{\boldsymbol{0}\}$ with the following semigroup operation
\begin{equation*}
(a,b)\cdot(c,d)=
\left\{
  \begin{array}{cl}
    (a,d),          & \hbox{if~~} b=c;\\
    \boldsymbol{0}, & \hbox{otherwise}
  \end{array}
\right.
\qquad \hbox{and} \qquad (a,b)\cdot\boldsymbol{0}=\boldsymbol{0}\cdot(a,b)=\boldsymbol{0}\cdot\boldsymbol{0}=\boldsymbol{0}, \quad \hbox{for all} \quad a,b,c,d\in\lambda,
\end{equation*}
is called the \emph{semigroup of $\lambda{\times}\lambda$-matrix units} \cite{Clifford-Preston-1961}. It is well known that $\mathscr{B}_{\lambda}$ is a combinatorial, congruence-free, primitive, completely $0$-simple inverse semigroup \cite{Lawson-1998, Petrich-1984}, and moreover $\mathscr{B}_{\lambda}$ is isomorphic to the semigroup $\mathscr{I}_\lambda^{1}$. By Proposition~4 of \cite{Gutik-Mykhalenych=2020} the semigroup $\boldsymbol{B}_{\omega}^{\mathscr{F}}$ is isomorphic to the semigroup of $\omega{\times}\omega$-matrix units $\mathscr{B}_{\omega}$ if and only if $\mathscr{F}=\{F,\varnothing\}$, where $F$ is a singleton subset of $\omega$.

For a non-zero cardinal $\lambda$ we denote by $\mathscr{S}_{\lambda}$ the group of bijective transformations of $\lambda$ and by $\mathscr{I\!\!T}\!_{\lambda}$ the semigroup of injective transformation of $\lambda$.

\begin{theorem}\label{theorem-3.1}
The semigroup $\mathfrak{End}^{\mathrm{inj}}(\mathscr{B}_{\lambda})$ of injective endomorphisms of $\mathscr{B}_{\lambda}$ is isomorphic to $\mathscr{I\!\!T}\!_{\lambda}$, and moreover the group $\mathfrak{Aut}(\mathscr{B}_{\lambda})$ of automorphisms of $\mathscr{B}_{\lambda}$ is isomorphic to $\mathscr{S}_{\lambda}$.
\end{theorem}

\begin{proof}
Let $\mathfrak{e}$  be an injective endomorphism of $\mathscr{B}_{\lambda}$. Then $(\boldsymbol{0})\mathfrak{e}=\boldsymbol{0}$ and the restriction of $\mathfrak{e}$ onto $E(\mathscr{B}_{\lambda})\setminus\{\boldsymbol{0}\}$ is an injection, i.e., there exists an injective transformation $\mathfrak{i}_{\mathfrak{e}}\colon\lambda\to\lambda$ such that $(a,a)\mathfrak{e}=((a)\mathfrak{i}_{\mathfrak{e}},(a)\mathfrak{i}_{\mathfrak{e}})$ for any $a\in\lambda$. It is obvious that $\mathfrak{i}_{\mathfrak{e}}\in \mathscr{I\!\!T}\!_{\lambda}$. Since the composition $\mathfrak{e}_1\circ\mathfrak{e}_2$ of two injective endomorphisms $\mathfrak{e}_1$ and $\mathfrak{e}_2$ of $\mathscr{B}_{\lambda}$ is an injective endomorphism,
\begin{equation*}
(a,a)(\mathfrak{e}_1\circ\mathfrak{e}_2)=((a)\mathfrak{i}_{\mathfrak{e}_1},(a)\mathfrak{i}_{\mathfrak{e_1}})\mathfrak{e}_2= (((a)\mathfrak{i}_{\mathfrak{e}_1})\mathfrak{i}_{\mathfrak{e}_2},((a)\mathfrak{i}_{\mathfrak{e_1}})\mathfrak{i}_{\mathfrak{e}_2}),
\end{equation*}
and hence $\mathfrak{i}_{\mathfrak{e}_1\circ\mathfrak{e}_2}=\mathfrak{i}_{\mathfrak{e_1}}\circ\mathfrak{i}_{\mathfrak{e}_2}$ is an injective map of $\lambda$. This implies that the such defined map $\mathfrak{J}\colon \mathfrak{End}^{\mathrm{inj}}(\mathscr{B}_{\lambda})\to \mathscr{I\!\!T}\!_{\lambda}$, $\mathfrak{e}\mapsto \mathfrak{i}_{\mathfrak{e}}$ is a homomorphism. Next we shall show that the homomorphism $\mathfrak{J}$ is surjective. Fix an arbitrary injective map $\mathfrak{i}\colon \lambda\to\lambda$. We claim that the mapping $\mathfrak{e}_{\mathfrak{i}}\colon \mathscr{B}_{\lambda}\to\mathscr{B}_{\lambda}$ by the formulae
\begin{equation*}
 (a,b)\mathfrak{e}_{\mathfrak{i}}=((a)\mathfrak{i},(b)\mathfrak{i}) \quad \hbox{for all} \quad a,b\in\lambda \qquad \hbox{and} \qquad (\boldsymbol{0})\mathfrak{e}_{\mathfrak{i}}=\boldsymbol{0},
\end{equation*}
is an injective endomorphism of the semigroup $\mathscr{B}_{\lambda}$. Indeed, since the mapping $\mathfrak{i}\colon \lambda\to\lambda$ is injective,
\begin{align*}
  (a,b)\mathfrak{e}_{\mathfrak{i}}\cdot(c,d)\mathfrak{e}_{\mathfrak{i}}&=((a)\mathfrak{i},(b)\mathfrak{i})\cdot ((c)\mathfrak{i},(d)\mathfrak{i})= \\
   & =
\left\{
  \begin{array}{cl}
    ((a)\mathfrak{i},(d)\mathfrak{i}),          & \hbox{if~~} (b)\mathfrak{i}=(c)\mathfrak{i};\\
    \boldsymbol{0}, & \hbox{otherwise}
  \end{array}
\right.=
\\
   & =
\left\{
  \begin{array}{cl}
    (a,d)\mathfrak{e}_{\mathfrak{i}},          & \hbox{if~~} b=c;\\
    \boldsymbol{0}, & \hbox{otherwise}
  \end{array}
\right.=\\
& =((a,b)\cdot(c,d))\mathfrak{e}_{\mathfrak{i}},
\end{align*}
and
\begin{align*}
  (a,b)\mathfrak{e}_{\mathfrak{i}}\cdot(\boldsymbol{0})\mathfrak{e}_{\mathfrak{i}}&= (a,b)\mathfrak{e}_{\mathfrak{i}}\cdot\boldsymbol{0} =
\boldsymbol{0} =(\boldsymbol{0})\mathfrak{e}_{\mathfrak{i}}=((a,b)\cdot\boldsymbol{0})\mathfrak{e}_{\mathfrak{i}};\\
  (\boldsymbol{0})\mathfrak{e}_{\mathfrak{i}}\cdot(a,b)\mathfrak{e}_{\mathfrak{i}}&= \boldsymbol{0}\cdot (a,b)\mathfrak{e}_{\mathfrak{i}}=
\boldsymbol{0} =(\boldsymbol{0})\mathfrak{e}_{\mathfrak{i}}=(\boldsymbol{0}\cdot(a,b))\mathfrak{e}_{\mathfrak{i}};\\
  (\boldsymbol{0})\mathfrak{e}_{\mathfrak{i}}\cdot(\boldsymbol{0})\mathfrak{e}_{\mathfrak{i}}&=\boldsymbol{0}\cdot \boldsymbol{0}= \boldsymbol{0}= (\boldsymbol{0})\mathfrak{e}_{\mathfrak{i}}=(\boldsymbol{0}\cdot \boldsymbol{0})\mathfrak{e}_{\mathfrak{i}},
\end{align*}
and hence $\mathfrak{e}_{\mathfrak{i}}$ is an  endomorphism of $\mathscr{B}_{\lambda}$. It is obvious that the injectivity of $\mathfrak{i}$ implies that the endomorphism $\mathfrak{e}_{\mathfrak{i}}$ is injective, too.

Simple verifications show that if $\mathfrak{e}$  be an automorphism of $\mathscr{B}_{\lambda}$ then the mapping $\mathfrak{i}_{\mathfrak{e}}\colon\lambda\to\lambda$ is bijective, and the bijectivity of the mapping $\mathfrak{i}\colon \lambda\to\lambda$ implies that
$\mathfrak{e}_{\mathfrak{i}}$ is an  automorphism of $\mathscr{B}_{\lambda}$. This completes the proof of the last statement.
\end{proof}

Recall \cite{Clifford-Preston-1961}, a semigroup $S$ is said to be \emph{left} (\emph{right}) \emph{cancellative} if for all $a,b,c\in S$, the equality $ab=ac$ ($ba=ca$) implies $b=c$. We remark that simple verification show that the semigroup $\mathscr{I\!\!T}\!_{\lambda}$ (and hence $\mathfrak{End}^{\mathrm{inj}}(\mathscr{B}_{\lambda})$) is left cancellative, but $\mathscr{I\!\!T}\!_{\lambda}$ is not right cancellative.

It is well known that the semigroup $\mathscr{B}_{\lambda}$ of $\lambda{\times}\lambda$-matrix units is congruence-free, i.e., $\mathscr{B}_{\lambda}$ has only two congruence: the identity and the universal congruence. This implies that every endomorphism of $\mathscr{B}_{\lambda}$ is either injective (i.e., is an isomorphism ``into'') or annihilating.

By $\mathfrak{End}^{\mathrm{ann}}(\mathscr{B}_{\lambda})$ we denote the semigroup of all annihilating endomorphisms of $\mathscr{B}_{\lambda}$.

It is obvious that for every annihilating endomorphism $\mathfrak{a}$ of $\mathscr{B}_{\lambda}$ there exits an idempotent $x\in\mathscr{B}_{\lambda}$ such that $(y)\mathfrak{a}=x$ for all $y\in\mathscr{B}_{\lambda}$, and later such endomorphism we denote by $\mathfrak{a}_{x}$. This implies that
\begin{equation*}
  \mathfrak{End}^{\mathrm{ann}}(\mathscr{B}_{\lambda})=\left\{\mathfrak{a}_{\boldsymbol{0}}\right\}\cup\left\{\mathfrak{a}_{(a,a)}\colon a\in \lambda\right\}.
\end{equation*}
It is obvious that $\mathfrak{End}^{\mathrm{ann}}(\mathscr{B}_{\lambda})$ is a right zero semigroup, $\mathfrak{End}^{\mathrm{ann}}(\mathscr{B}_{\lambda})$ is left simple and hence it is simple.

For any $\mathfrak{e}\in\mathfrak{End}^{\mathrm{inj}}(\mathscr{B}_{\lambda})$ and  $\mathfrak{a}_{x}\in\mathfrak{End}^{\mathrm{ann}}(\mathscr{B}_{\lambda})$ we have that
\begin{equation*}
  \mathfrak{e}\circ \mathfrak{a}_{x}=\mathfrak{a}_{x} \qquad \hbox{and} \qquad \mathfrak{a}_{x}\circ \mathfrak{e}=\mathfrak{a}_{(x)\mathfrak{e}}.
\end{equation*}

The above arguments we summarize in the following theorem:

\begin{theorem}\label{theorem-3.2}
The semigroup $\mathfrak{End}(\mathscr{B}_{\lambda})$ of all endomorphisms of the semigroup of $\lambda{\times}\lambda$-matrix units $\mathscr{B}_{\lambda}$ is the union of the semigroups $\mathfrak{End}^{\mathrm{inj}}(\mathscr{B}_{\lambda})$ and $\mathfrak{End}^{\mathrm{ann}}(\mathscr{B}_{\lambda})$. Moreover, $\mathfrak{End}^{\mathrm{inj}}(\mathscr{B}_{\lambda})$ a left cancellative semigroup and  $\mathfrak{End}^{\mathrm{ann}}(\mathscr{B}_{\lambda})$ is the minimal ideal of $\mathfrak{End}(\mathscr{B}_{\lambda})$ which is a right zero semigroup.
\end{theorem}
%%%%%%%%%%%%%%%%%%%%%%%%%%%%%%%%%%%%%%%%%%%%%%%%%%%%%
%\section*{Acknowledgements}

%The author acknowledge Alex Ravsky and the referee for their comments and suggestions.
%%%%%%%%%%%%%%%%%%%%%%%%%%%%%%%%%%%%%%%%%%%%%%%%%%%%%%%%%%%%

\end{document}